\documentclass[11pt, twoside, leqno]{article}

\usepackage{amssymb}
\usepackage{amsmath}
\usepackage{amsthm}
\usepackage{color}
\usepackage{mathrsfs}
\usepackage{txfonts}

\usepackage{indentfirst}

\allowdisplaybreaks

\newtheorem{theorem}{Theorem}[section]
\newtheorem{lemma}[theorem]{Lemma}

\theoremstyle{definition}
\newtheorem{remark}[theorem]{Remark}

\renewcommand{\appendix}{\par
   \setcounter{section}{0}%
   \setcounter{subsection}{0}%
   \setcounter{subsubsection}{0}%
   \gdef\thesection{\@Alph\c@section}%
   \gdef\thesubsection{\@Alph\c@section.\@arabic\c@subsection}%
   \gdef\theHsection{\@Alph\c@section.}%
   \gdef\theHsubsection{\@Alph\c@section.\@arabic\c@subsection}%
   \csname appendixmore\endcsname
 }

\numberwithin{equation}{section}

\usepackage{graphicx} 

\begin{document}

\arraycolsep=1pt

\title{\bf\Large Numerical Approximations and Error Analysis of the Cahn-Hilliard Equation with Dynamic Boundary Conditions
\footnotetext{\hspace{-0.35cm} 2010 {\it
Mathematics Subject Classification}. 65M12; 65M06; 65N12; 65M22.
\endgraf {\it Key words and phrases.}
Cahn-Hilliard equation; Dynamic boundary conditions; Error estimates; Linear numerical
scheme; Energy stability.
}}
\author{Xuelian Bao\footnote{Corresponding author, School of Mathematical Sciences, Beijing Normal University, Beijing 100875, China (e-mail: xlbao@mail.bnu.edu.cn).}, Hui Zhang\footnote{Laboratory of Mathematics and Complex Systems, Ministry of Education and School of Mathematical Sciences, Beijing Normal University, Beijing 100875, China.}}
\date{}
\maketitle

\vspace{-0.8cm}

\begin{center}
\begin{minipage}{13cm}
{\small {\bf Abstract}\quad
We consider the numerical approximations of the Cahn-Hilliard equation with dynamic boundary conditions (C. Liu et. al., Arch. Rational Mech. Anal., 2019). We propose a first-order in time, linear and energy stable numerical scheme, which is based on the stabilized linearly implicit approach. The energy stability of the scheme is proved and the semi-discrete-in-time error estimates are carried out. Numerical experiments, including the comparison with the former work, the accuracy tests with respect to the time step size and the shape deformation of a droplet, are performed to validate the accuracy and the stability of the proposed scheme.
}
\end{minipage}
\end{center}

\section{Introduction}\label{s1}

The Cahn-Hilliard equation is one of the most fundamental models which describe the phase separation processes of binary mixtures. The classical Cahn-Hilliard equation, first introduced in \cite{CH1958}, can be written as follows:
\begin{equation}\label{CH}
\left\{
\begin{aligned}
&\phi_t=\Delta\mu, &\mbox{in}\  \Omega\times(0,T),\\
&\mu=-\varepsilon\Delta\phi+\frac{1}{\varepsilon}F'(\phi), &\mbox{in}\ \Omega\times(0,T),
\end{aligned}
\right.
\end{equation}
where $T\in (0,\infty)$ and $\Omega\subseteq\mathbb{R}^{d}$ ($d=2,3$) is a bounded domain with the smooth boundary $\Gamma=\partial\Omega$. The phase-field order parameter $\phi$
represents the difference of two local relative concentrations to describe the binary mixtures. In the domain $\Omega$, $\phi=\pm1$ correspond to the pure phases of the materials, which are separated by a interfacial region whose thickness is proportional to the parameter $\varepsilon$. $\mu$ represents the chemical potential in $\Omega$, which can be expressed as the Fr\'{e}chet derivative of the bulk free energy:
\begin{equation}\label{Ebulk}
E^{bulk}(\phi)=\int_{\Omega}\frac{\varepsilon}2|\nabla\phi|^2+\frac{1}{\varepsilon}F(\phi)\mbox{d}x,
\end{equation}
where $F$ stands for the bulk potential, which usually has a double-well structure with two minima at -1 and 1 and a local unstable maximum at 0. A classical choice is the regular double-well potential
\begin{equation}\label{classicalF}
F(x)=\frac{1}4(x^2-1)^2, \qquad x\in \mathbb{R}.
\end{equation}

When the time-evolution of $\phi$ is confined in a bounded domain, suitable boundary conditions should be considered for the system \eqref{CH}.
The homogeneous Neumann conditions are the classical boundary conditions:
\begin{equation}\label{BCmu}
\partial_{\mathbf{n}}\mu=0, \quad \mbox{on}\  \Gamma\times(0,T),
\end{equation}
\begin{equation}\label{BCphi}
\partial_{\mathbf{n}}\phi=0, \quad \mbox{on}\  \Gamma\times(0,T),
\end{equation}
where $\mathbf{n}=\mathbf{n}(\mathbf{x})$ denotes the unit outer normal vector and $\partial_{\mathbf{n}}$ denotes the outward normal derivative on $\Gamma$.
The Cahn-Hilliard equation with the boundary conditions \eqref{BCmu} and \eqref{BCphi} can be viewed as an $H^{-1}$-gradient flow of the bulk free energy.

The no-flux boundary condition \eqref{BCmu} guarantees the conservation of mass in the
bulk (i.e., in $\Omega$):
\begin{equation}
\int_{\Omega}\phi(t) \mbox{d}x=\int_{\Omega}\phi(0) \mbox{d}x, \quad t\in[0,T].
\end{equation}
Moreover, the boundary conditions \eqref{BCmu} and \eqref{BCphi} imply that the bulk
free energy $E^{bulk}$ (Eq. \eqref{Ebulk}) is decreasing with respect to time, namely,
\begin{equation}
\frac{d}{dt} E^{bulk}(\phi(t))+\int_{\Omega}|\nabla\mu|^2 \mbox{d}x=0, \quad t\in(0,T).
\end{equation}

However, the Cahn-Hilliard equation with homogeneous Neumann conditions neglects the effects of the boundary to the bulk dynamics. Thus, it is not suitable for some applications (for instance, hydrodynamic applications such as contact line problems).
In order to describe the effective interactions between the solid wall and the binary mixture, physicists added the suitable surface free energy functional into the system \cite{fischer1997, fischer1998, Kenzler2001}:
\begin{equation}\label{Etotal}
E^{total}(\phi)=E^{bulk}(\phi)+E^{surf}(\phi),
\end{equation}
\begin{equation}\label{Esurf}
E^{surf}(\phi)=\int_{\Gamma}\frac{\delta\kappa}2|\nabla_{\Gamma}\phi|^2+\frac{1}{\delta}G(\phi)\mbox{d}S,
\end{equation}
where $\nabla_{\Gamma}$ represents the tangential or surface gradient operator on $\Gamma$, $G$ is the surface potential, the parameter $\kappa$ is related to the surface diffusion and $\delta$ denotes the thickness of the interfacial region on $\Gamma$. When $\kappa=0$, it is related to the moving contact line problem \cite{Thompson1989}.
Recently, for the total free energy \eqref{Etotal}, various dynamic boundary conditions for the Cahn-Hilliard equation have been proposed and investigated, see for instance, see \cite{Kenzler2001, GMS2011, liuwu2019, knopf2019, knopf2020}, and references therein.

In the present work, the Cahn-Hilliard equation with the dynamic boundary conditions, which was derived by an energetic variational approach by Liu and Wu (Liu-Wu model, for short) \cite{liuwu2019}, is considered.
It reads as follows:
\begin{equation}\label{CHLW}
\left\{
\begin{aligned}
&\phi_t=\Delta\mu, &\mbox{in}\  \Omega\times(0,T),\\
&\mu=-\varepsilon\Delta\phi+\frac{1}{\varepsilon}F'(\phi), &\mbox{in}\ \Omega\times(0,T),\\
&\partial_{\mathbf{n}}\mu=0, &\mbox{on}\  \Gamma\times(0,T),\\
&\phi|_{\Gamma}=\psi, &\mbox{on}\  \Gamma\times(0,T),\\
&\psi_t=\Delta_{\Gamma}\mu_{\Gamma}, &\mbox{on}\  \Gamma\times(0,T),\\
&\mu_{\Gamma}=-\delta\kappa\Delta_{\Gamma}\psi+\frac{1}{\delta}G'(\psi)+\varepsilon\partial_{\mathbf{n}}\phi
&\mbox{on}\  \Gamma\times(0,T),
\end{aligned}
\right.
\end{equation}
where $\Delta_{\Gamma}$ denotes the Laplace-Beltrami operator on $\Gamma$.
The dynamic boundary conditions (with $\delta$, $\kappa>0$) turns out to be a surface Cahn-Hilliard type equation for the trace of $\phi$ on $\Gamma$, coupled with the bulk evolution in terms of $\partial_{\mathbf{n}}\phi$.
The existence and uniqueness of weak and strong solutions of the Liu-Wu model have been established in \cite{liuwu2019}. A different approach to construct the weak solutions of the Liu-Wu model is proposed in \cite{Garcke2020}.

The numerical approximations of the Cahn-Hilliard equation and its variants have been intensively investigated. The stabilized linearly implicit approach \cite{he2007, xu2020}, the approaches based on the convex-concave splitting \cite{shen2012, grun2013}, the invariant energy quadratization \cite{ieq1, ieq2} and
the scalar auxiliary variable (SAV) \cite{sav} method are efficient techniques for the time discretization.
Recently, there have been numerous contributions on the numerical approximation of the
Cahn-Hilliard equation with dynamic boundary conditions \cite{Cherfils2010,Cherfils2014,Israel2015,Fukao2017,bachelor}. For the numerical approximations of Liu-Wu model,
the first finite element scheme was proposed in \cite{bachelor} and the corresponding numerical results were presented in \cite{Garcke2020},
where the straightforward discretization based on piecewise linear finite element functions was utilized to simulate Liu-Wu model, and the corresponding nonlinear system was solved by Newton's method. A recent contribution on the numerical analysis for the Liu-Wu model can be found in \cite{Metzger2019}, where a different discrete scheme was proposed and the connection between $\phi$ and the chemical potentials was investigated.
However, the backward implicit Euler method was used for time discretization in the above discrete schemes, which lead to nonlinear systems at each time step.

In the present work, a first-order in time, linear and energy stable scheme for solving the Liu-Wu model is proposed based on the stabilized linearly implicit approach.
At each time step, one only needs to solve one linear equation and thus, the scheme is highly efficient. The energy stability of the scheme is proved and various numerical simulations in two-dimensional spaces are performed to validate the accuracy and stability of the scheme by comparing with the former work.
The error estimates in semi-discrete-in-time for the scheme are also carried out.
To the best of the authors' knowledge, the proposed scheme in this paper is the first linear and energy stable scheme for solving the Liu-Wu model and it is the first work to give the semi-discrete-in-time error estimates for the model.

The rest of the paper is organized as follows.
In Section 2, we first recall some notions and notation appearing in this article. In Section 3, a simple derivation of Liu-Wu model and the stabilized scheme with the energy stability are derived. In Section 4, we construct the error estimates. The accuracy tests and numerical examples are presented in Section 5. Finally, some concluding remarks are presented in Section 6.

\section{Preliminaries}\label{s2}

Before giving the stabilized scheme and the corresponding error analysis, we make some definitions in this section.
The norm and inner product of $L^2(\Omega)$ and $L^2(\Gamma)$ are denoted by $\|\cdot\|_{\Omega}$, $(\cdot,\cdot)_{\Omega}$ and $\|\cdot\|_{\Gamma}$, $(\cdot,\cdot)_{\Gamma}$ respectively. The usual norm in $H^k(\Omega)$ and $H^k(\Gamma)$ are denoted by $\|\cdot\|_{H^k(\Omega)}$ and $\|\cdot\|_{H^k(\Gamma)}$ respectively.

We consider a finite time interval $[0,T]$ and a domain $\Omega\subset\mathbb{R}^d$ ($d=2, 3$), which is a bounded domain with sufficient smooth boundary $\Gamma=\partial\Omega$ and $\mathbf{n}=\mathbf{n}(x)$ is the unit outer normal vector on $\Gamma$.

Let $\tau$ be the time step size. For a sequence of functions $f^0, f^1, \ldots, f^N$ in some Hilbert space $E$, we denote the sequence by $\{f_{\tau}\}$ and define the following discrete norm for $\{f_{\tau}\}$:
\begin{equation}\label{discretenorm}
\|f_{\tau}\|_{l^{\infty}(E)}=\max_{0\leq n\leq N}\bigg{(}\|f^n\|_E\bigg{)}.
\end{equation}
We denote by $C$ a generic constant that is independent of $\tau$ but possibly depends on the data and the solution, and use $f\lesssim g$ to say that there is a generic constant $C$ such that $f\leqslant C g$.

\section{Derivation of the Cahn-Hilliard equation with dynamic boundary conditions and its numerical scheme}\label{s3}

In this section, we first propose a simple derivation of the Liu-Wu model, indicating that it satisfies the energy dissipation law and mass conservation.

Since $\phi$ is the phase-field order parameter in the bulk, denote its trace $\psi:=\phi|_{\Gamma}$ as the order parameter on the boundary.
When the mass conservation holds true in the bulk $\Omega$ and on the boundary $\Gamma$ respectively, $\phi$ and $\psi$ satisfy the following continuity equations \cite{Forsterthesis}:
\begin{equation}\label{mass1}
\begin{aligned}
&\phi_t+\nabla\cdot(\phi\mathbf{u})=0, \qquad \mbox{in} \ \Omega\times(0,T),\\
&\psi_t+\nabla_{\Gamma}\cdot(\phi\mathbf{v})=0, \qquad \mbox{on}\  \Gamma\times(0,T),
\end{aligned}
\end{equation}
where $\mathbf{u}$ is the microscopic effective velocity and $\mathbf{v}$ is the microscopic effective tangential velocity field on the boundary.
Assume that there is no mass exchange between the bulk and the boundary, $\mathbf{u}$ satisfies the following boundary condition:
\begin{equation}
\mathbf{u}\cdot\mathbf{n}=0, \qquad \mbox{on}\  \Gamma\times(0,T).
\end{equation}
Since the boundary is closed, there is no need to impose any boundary condition
on $\mathbf{v}$.

For the total free energy \eqref{Etotal}, we consider the following energy dissipation law:
\begin{equation}\label{disslaw}
\frac{d}{dt} E^{total}(t)=-\mathcal{D}^{total}(t),
\end{equation}
which is based on the first and second laws of thermodynamics \cite{Hyon2010}. And the rate of energy dissipation $\mathcal{D}^{total}(t)$ also consists of two contributions from the bulk and the boundary, namely,
\begin{equation}\label{dtotal}
\mathcal{D}^{total}(t)=\mathcal{D}^{bulk}(t)+\mathcal{D}^{surf}(t).
\end{equation}
Assume that
\begin{equation}
\mathcal{D}^{bulk}(t)=\int_{\Omega} \phi^2 |\mathbf{u}|^2 \mbox{d}x, \qquad \mathcal{D}^{surf}(t)=\int_{\Gamma} \psi^2 |\mathbf{v}|^2 \mbox{d}S,
\end{equation}
and substitute \eqref{Ebulk} and \eqref{Esurf} into \eqref{disslaw}, we obtain
\begin{equation}\label{disslaw2}
\begin{aligned}
&\frac{d}{dt} \bigg{[} \int_{\Omega} \frac{\varepsilon}2|\nabla\phi|^2+\frac{1}{\varepsilon}F(\phi)\mbox{d}x+ \int_{\Gamma}\frac{\delta\kappa}2|\nabla_{\Gamma}\phi|^2+\frac{1}{\delta}G(\phi)\mbox{d}S\bigg{]}\\
&=-  \int_{\Omega} \phi^2 |\mathbf{u}|^2 \mbox{d}x- \int_{\Gamma} \psi^2 |\mathbf{v}|^2 \mbox{d}S.
\end{aligned}
\end{equation}
The left part of \eqref{disslaw2} can be written as
\begin{equation}\label{disslaw3}
\begin{aligned}
&\frac{d}{dt} \bigg{[} \int_{\Omega} \frac{\varepsilon}2|\nabla\phi|^2+\frac{1}{\varepsilon}F(\phi)\mbox{d}x+ \int_{\Gamma}\frac{\delta\kappa}2|\nabla_{\Gamma}\phi|^2+\frac{1}{\delta}G(\phi)\mbox{d}S\bigg{]}\\
&=\int_{\Omega} [-\varepsilon\Delta\phi+\frac{1}{\varepsilon}F'(\phi)]\phi_t \mbox{d}x
+\int_{\Gamma} [\varepsilon\partial_{\mathbf{n}}\phi-\delta\kappa\Delta_{\Gamma}\psi+\frac{1}{\delta}G'(\psi)]
\psi_t \mbox{d}S\\
&=\int_{\Omega} \phi \nabla[-\varepsilon\Delta\phi+\frac{1}{\varepsilon}F'(\phi)]\cdot\mathbf{u}\mbox{d}x
+\int_{\Gamma} \psi \nabla_{\Gamma} [\varepsilon\partial_{\mathbf{n}}\phi-\delta\kappa\Delta_{\Gamma}\psi+\frac{1}{\delta}G'(\psi)]
 \cdot\mathbf{v}\mbox{d}S\\
&=\int_{\Omega} \phi \nabla \mu\cdot\mathbf{u}\mbox{d}x+\int_{\Gamma} \psi \nabla_{\Gamma}
[\varepsilon\partial_{\mathbf{n}}\phi+\mu_s]\cdot\mathbf{v}\mbox{d}S.
\end{aligned}
\end{equation}
Here, $\mu$ and $\mu_s$ are the chemical potentials in $\Omega$ and on $\Gamma$, respectively, which can be expressed as the Fr\'{e}chet derivative of the bulk free energy $E^{bulk}$ and the surface free energy $E^{surf}$, namely,
\begin{equation}
\begin{aligned}
&\mu=\frac{\delta E^{bulk}}{\delta \phi}=-\varepsilon\Delta\phi+\frac{1}{\varepsilon}F'(\phi)\\
&\mu_s=\frac{\delta E^{surf}}{\delta \psi}=-\delta\kappa\Delta_{\Gamma}\psi+\frac{1}{\delta}G'(\psi).
\end{aligned}
\end{equation}
From \eqref{disslaw2} and \eqref{disslaw3}, we obtain
\begin{equation}\label{disslaw4}
\begin{aligned}
&\phi \mathbf{u}=-\nabla \mu,\\
&\psi\mathbf{v}=-\nabla_{\Gamma}(\varepsilon\partial_{\mathbf{n}}\phi+\mu_s).
\end{aligned}
\end{equation}
Substitute \eqref{disslaw4} into \eqref{mass1}, we obtain the Liu-Wu model (Eq. \eqref{CHLW}) with $\mu_{\Gamma}=\varepsilon\partial_{\mathbf{n}}\phi+\mu_s$.

Obviously, the Liu-Wu model satisfies the mass conservation law in the bulk and on the boundary:
\begin{equation}\label{massconservation}
\int_{\Omega}\phi(t) \mbox{d}x=\int_{\Omega}\phi(0) \mbox{d}x \quad \mbox{and} \quad \int_{\Gamma}\psi(t) \mbox{d}S=\int_{\Gamma}\psi(0) \mbox{d}S, \quad t\in[0,T],
\end{equation}
Moreover, from the energy dissipation law, it is easy to see that the total free energy $E^{total}(\phi,\psi)=E^{bulk}(\phi)+E^{surf}(\psi)$ is decreasing in time:
\begin{equation}\label{energy3}
\frac{\mbox{d}}{\mbox{d}t} E^{total}(\phi,\psi)=-\|\nabla\mu\|_{\Omega}^2
-\|\nabla_{\Gamma}\mu_{\Gamma}\|_{\Gamma}^2\leq 0.
\end{equation}

Now we present the stabilized scheme for the Cahn-Hilliard equation with dynamic boundary conditions (namely, \eqref{CHLW}).
The scheme can be written as follows,
\begin{eqnarray}\label{SIscheme1}
&&\frac{\phi^{n+1}-\phi^{n}}{\tau}=\Delta\mu^{n+1}, \quad \mbox{in}\  \Omega,\\
\label{SIscheme2}
&&\mu^{n+1}=-\varepsilon\Delta\phi^{n+1}+\frac{1}{\varepsilon}F'(\phi^{n})+s_1(\phi^{n+1}-\phi^{n}), \quad \mbox{in}\ \Omega,\\
\label{SIscheme3}
&&\partial_{\mathbf{n}}\mu^{n+1}=0, \quad \mbox{on}\  \Gamma,\\
\label{SIscheme4}
&&\phi^{n+1}|_{\Gamma}=\psi^{n+1}, \quad \mbox{on}\  \Gamma,\\
\label{SIscheme5}
&&\frac{\psi^{n+1}-\psi^{n}}{\tau}=\Delta_{\Gamma}\mu_{\Gamma}^{n+1}, \quad \mbox{on}\  \Gamma,\\
\label{SIscheme6}
&&\mu_{\Gamma}^{n+1}=-\delta\kappa\Delta_{\Gamma}\psi^{n+1}+\frac{1}{\delta}G'(\psi^{n})
+\varepsilon\partial_{\mathbf{n}}\phi^{n+1}+s_2(\psi^{n+1}-\psi^{n}), \quad \mbox{on}\  \Gamma.
\end{eqnarray}
Here, $T$ is an arbitrary and fixed time, $N$ is the number of time steps and $\tau=T/N$ is the step size.
We have the energy stability as follows.
\begin{theorem}\label{ener_stable}
Assume that $s_1\geq\frac{1}{2\varepsilon} \max_{\xi\in\mathbb{R}} F''(\xi)$, $s_2\geq\frac{1}{2\delta} \max_{\eta\in\mathbb{R}} G''(\eta)$, the scheme \eqref{SIscheme1}-\eqref{SIscheme6} is energy stable in the sense that
\begin{equation}
\begin{aligned}
\frac{E(\phi^{n+1},\psi^{n+1})-E(\phi^{n},\psi^{n})}{\tau}&\leq-\|\nabla\mu^{n+1}\|^2_{\Omega}
-\|\nabla_{\Gamma}\mu_{\Gamma}^{n+1}\|^2_{\Gamma},
\end{aligned}
\end{equation}
where
\begin{equation}\label{discreteenergy}
E(\phi^{n},\psi^{n})=\int_{\Omega}\frac{1}{\varepsilon} F(\phi^{n})+\frac{\varepsilon}{2} |\nabla\phi^{n}|^2\mbox{d}x+\int_{\Gamma} \frac{1}{\delta} G(\psi^{n})+\frac{\delta\kappa}{2}|\nabla_{\Gamma}\psi^{n}|^2\mbox{d}S
\end{equation}
\end{theorem}

\begin{proof}
By taking inner product of \eqref{SIscheme1} with $\mu^{n+1}$ in $\Omega$, we have
\begin{equation}\label{proof1}
(\frac{\phi^{n+1}-\phi^{n}}{\tau}, \mu^{n+1})_{\Omega}=(\Delta\mu^{n+1}, \mu^{n+1})_{\Omega}=-\|\nabla\mu^{n+1}\|^2_{\Omega}.
\end{equation}

By using \eqref{SIscheme2}, we have
\begin{equation}\label{proof2}
\begin{aligned}
(\frac{\phi^{n+1}-\phi^{n}}{\tau}, \mu^{n+1})_{\Omega}&=(\frac{\phi^{n+1}-\phi^{n}}{\tau}, -\varepsilon\Delta\phi^{n+1}
+\frac{1}{\varepsilon}F'(\phi^{n})+s_1(\phi^{n+1}-\phi^{n}))_{\Omega},
\end{aligned}
\end{equation}
\begin{equation}\label{proof3}
\begin{aligned}
(\frac{\phi^{n+1}-\phi^{n}}{\tau}, -\varepsilon\Delta\phi^{n+1})_{\Omega}&=
-\varepsilon(\partial_{\mathbf{n}}\phi^{n+1}, \frac{\phi^{n+1}-\phi^{n}}{\tau})_{\Gamma}+\varepsilon(\nabla\phi^{n+1},
\frac{\nabla\phi^{n+1}-\nabla\phi^{n}}{\tau})_{\Omega}.
\end{aligned}
\end{equation}

For the boundary integral term in \eqref{proof3}, by taking the inner product of \eqref{SIscheme5} with $\mu^{n+1}_{\Gamma}$ on $\Gamma$, we obtain
\begin{equation}\label{proof4}
(\frac{\psi^{n+1}-\psi^{n}}{\tau}, \mu^{n+1}_{\Gamma})_{\Gamma}=
(\Delta_{\Gamma} \mu^{n+1}_{\Gamma}, \mu^{n+1}_{\Gamma})_{\Gamma}
=-\|\nabla_{\Gamma} \mu^{n+1}_{\Gamma}\|^2_{\Gamma}.
\end{equation}

By using \eqref{SIscheme6}, we have
\begin{equation}\label{proof5}
\begin{aligned}
(\frac{\psi^{n+1}-\psi^{n}}{\tau}, \mu^{n+1}_{\Gamma})_{\Gamma}&=(\frac{\psi^{n+1}-\psi^{n}}{\tau},-\delta\kappa\Delta_{\Gamma}
\psi^{n+1}+\frac{1}{\delta}G'(\psi^{n})\\
&+\varepsilon\partial_{\mathbf{n}}\phi^{n+1}+s_2(\psi^{n+1}-\psi^{n}))_{\Gamma},
\end{aligned}
\end{equation}
\begin{equation}\label{proof6}
(\frac{\psi^{n+1}-\psi^{n}}{\tau}, -\delta\kappa\Delta_{\Gamma}\psi^{n+1})_{\Gamma}=
(\frac{\nabla_{\Gamma}\psi^{n+1}-\nabla_{\Gamma}\psi^{n}}{\tau}, \delta\kappa\nabla_{\Gamma}\psi^{n+1})_{\Gamma}.
\end{equation}
To handle the nonlinear term associated with $F'$ and $G'$ in \eqref{proof2} and \eqref{proof5}, we need the following identities
\begin{equation}\label{proof7}
\begin{aligned}
F'(\phi^{n})(\phi^{n+1}-\phi^{n})&=F(\phi^{n+1})-F(\phi^{n})-\frac{F''(\eta)}2(\phi^{n+1}-\phi^{n})^2,\\
G'(\phi^{n})(\phi^{n+1}-\phi^{n})&=G(\phi^{n+1})-G(\phi^{n})-\frac{G''(\zeta)}2(\phi^{n+1}-\phi^{n})^2.
\end{aligned}
\end{equation}

Combining the equations mentioned above, we get
$$
\begin{aligned}
&(\frac{\phi^{n+1}-\phi^{n}}{\tau}, \mu^{n+1})_{\Omega}+(\frac{\psi^{n+1}-\psi^{n}}{\tau}, \mu^{n+1}_{\Gamma})_{\Gamma}=-\|\nabla\mu^{n+1}\|^2_{\Omega}-\|\nabla_{\Gamma}\mu^{n+1}_{\Gamma}\|^2_{\Gamma},
\end{aligned}
$$
and
$$
\begin{aligned}
&(\frac{\phi^{n+1}-\phi^{n}}{\tau}, \mu^{n+1})_{\Omega}+(\frac{\psi^{n+1}-\psi^{n}}{\tau}, \mu^{n+1}_{\Gamma})_{\Gamma}\\
&=\varepsilon(\nabla\phi^{n+1},
\frac{\nabla\phi^{n+1}-\nabla\phi^{n}}{\tau})_{\Omega}
+\frac{1}{\varepsilon}(F'(\phi^{n}),\frac{\phi^{n+1}-\phi^{n}}{\tau})_{\Omega}
+\frac{s_1}{\tau}\|\phi^{n+1}-\phi^{n}\|^2_{\Omega}\\
&+(\delta\kappa\nabla_{\Gamma}\psi^{n+1},\frac{\nabla_{\Gamma}\psi^{n+1}-\nabla_{\Gamma}\psi^{n}}{\tau})_{\Gamma}
+\frac{1}{\delta}(G'(\psi^{n}),\frac{\psi^{n+1}-\psi^{n}}{\tau})_{\Gamma}
+\frac{s_2}{\tau}\|\psi^{n+1}-\psi^{n}\|^2_{\Gamma}\\
&=\varepsilon(\nabla\phi^{n+1},
\frac{\nabla\phi^{n+1}-\nabla\phi^{n}}{\tau})_{\Omega}+\frac{1}{\varepsilon}
(\frac{F(\phi^{n+1})-F(\phi^{n})}{\tau},1)_{\Omega}
-\frac{1}{2\varepsilon}(F''(\eta),\frac{(\phi^{n+1}-\phi^{n})^2}{\tau})_{\Omega}\\
&+\frac{s_1}{\tau}\|\phi^{n+1}-\phi^{n}\|^2_{\Omega}
+\delta\kappa(\nabla_{\Gamma}\psi^{n+1},\frac{\nabla_{\Gamma}\psi^{n+1}-\nabla_{\Gamma}\psi^{n}}{\tau})_{\Gamma}
+\frac{1}{\delta}(\frac{G(\psi^{n+1})-G(\psi^{n})}{\tau},1)_{\Gamma}\\
&-\frac{1}{2\delta}(G''(\zeta),\frac{(\psi^{n+1}-\psi^{n})^2}{\tau})_{\Gamma}
+\frac{s_2}{\tau}\|\psi^{n+1}-\psi^{n}\|^2_{\Gamma}\\
&=\frac{\varepsilon}{2\tau}(\|\nabla\phi^{n+1}\|^2_{\Omega}-
\|\nabla\phi^{n}\|^2_{\Omega}+\|\nabla\phi^{n+1}-\nabla\phi^{n}\|^2_{\Omega})
+\frac{1}{\varepsilon\tau}(F(\phi^{n+1})-F(\phi^{n}),1)_{\Omega}\\
&+\frac{1}{\tau}(s_1-\frac{1}{2\varepsilon}F''(\eta))\|\phi^{n+1}-\phi^{n}\|^2_{\Omega}
+\frac{\delta\kappa}{2\tau}(\|\nabla_{\Gamma}\psi^{n+1}\|^2_{\Gamma}\\
&-\|\nabla_{\Gamma}\psi^{n}\|^2_{\Gamma}+\|\nabla_{\Gamma}\psi^{n+1}-\nabla_{\Gamma}\psi^{n}\|^2_{\Gamma})
+\frac{1}{\delta\tau}(G(\psi^{n+1})-G(\psi^{n}),1)_{\Gamma}\\
&+\frac{1}{\tau}
(s_2-\frac{1}{2\delta}G''(\zeta))\|\psi^{n+1}-\psi^{n}\|^2_{\Gamma}\\
&=\frac{1}{\tau}[E(\phi^{n+1},\psi^{n+1})-E(\phi^{n},\psi^{n})]+\frac{\varepsilon}{2\tau}
\|\nabla\phi^{n+1}-\nabla\phi^{n}\|^2_{\Omega}\\
&+\frac{\delta\kappa}{2\tau}\|\nabla_{\Gamma}\psi^{n+1}-\nabla_{\Gamma}\psi^{n}\|^2_{\Gamma}
+\frac{1}{\tau}(s_1-\frac{1}{2\varepsilon}F''(\eta))\|\phi^{n+1}-\phi^{n}\|^2_{\Omega}\\
&+\frac{1}{\tau}(s_2-\frac{1}{2\delta}G''(\zeta))\|\psi^{n+1}-\psi^{n}\|^2_{\Gamma}.
\end{aligned}
$$
Thus, we have
$$
\begin{aligned}
&\frac{1}{\tau}[E(\phi^{n+1},\psi^{n+1})-E(\phi^{n},\psi^{n})]+\frac{\varepsilon}{2\tau}
\|\nabla\phi^{n+1}-\nabla\phi^{n}\|^2_{\Omega}\\
&+\frac{\delta\kappa}{2\tau}\|\nabla_{\Gamma}\psi^{n+1}-\nabla_{\Gamma}\psi^{n}\|^2_{\Gamma}
+\frac{1}{\tau}(s_1-\frac{1}{2\varepsilon}F''(\eta))\|\phi^{n+1}-\phi^{n}\|^2_{\Omega}\\
&+\frac{1}{\tau}(s_2-\frac{1}{2\delta}G''(\zeta))\|\psi^{n+1}-\psi^{n}\|^2_{\Gamma}
=-\|\nabla\mu^{n+1}\|^2_{\Omega}-\|\nabla_{\Gamma}\mu^{n+1}_{\Gamma}\|^2_{\Gamma}\leq0.
\end{aligned}
$$
Therefore, under the conditions that
$$s_1\geq\frac{1}{2\varepsilon} \max_{\xi\in\mathbb{R}} F''(\xi)$$
and
$$s_2\geq\frac{1}{2\delta} \max_{\eta\in\mathbb{R}} G''(\eta),$$
we have
$$
\frac{1}{\tau}[E(\phi^{n+1},\psi^{n+1})-E(\phi^{n},\psi^{n})]\leq0,
$$
namely, the scheme (\ref{SIscheme1})-(\ref{SIscheme6}) is energy stable.
\end{proof}

\section{Error estimates for the stabilized semi-discrete scheme}\label{s4}

In this section, we establish the error estimates for the phase functions $\phi$ and $\psi$ for the stabilized scheme \eqref{SIscheme1}-\eqref{SIscheme6}.

Assume that the Lipschitz properties hold for the derivatives of $F'$ and $G'$,
\begin{equation}
\max_{\phi\in\mathbb{R}} |F^{''}(\phi)|\leq L_1,
\end{equation}
\begin{equation}
\max_{\psi\in\mathbb{R}} |G^{''}(\psi)|\leq L_2,
\end{equation}
which are necessary for error estimates.

The PDE system \eqref{CHLW} can be rewritten as the following truncated form,
\begin{eqnarray}\label{truncatePDE}
&&\frac{\phi(t^{n+1})-\phi(t^{n})}{\tau}=\Delta\mu(t^{n+1})+R_{\phi}^{n+1}, \quad \mbox{in}\  \Omega,\\
\label{truncatePDE2}
&&\mu(t^{n+1})=-\varepsilon\Delta\phi(t^{n+1})+\frac{1}{\varepsilon}F'(\phi(t^{n}))
+s_1(\phi(t^{n+1})-\phi(t^{n}))+R_{\mu}^{n+1}, \quad \mbox{in}\ \Omega,\\
\label{truncatePDE3}
&&\partial_{\mathbf{n}}\mu(t^{n+1})=0, \quad \mbox{on}\  \Gamma,\\
\label{truncatePDE4}
&&\phi(t^{n+1})|_{\Gamma}=\psi(t^{n+1}), \quad \mbox{on}\  \Gamma,\\
\label{truncatePDE5}
&&\frac{\psi(t^{n+1})-\psi(t^{n})}{\tau}=\Delta_{\Gamma}\mu_{\Gamma}(t^{n+1})+R_{\psi}^{n+1}, \quad \mbox{on}\  \Gamma,\\
\label{truncatePDE6}
&&\mu_{\Gamma}(t^{n+1})=-\delta\kappa\Delta_{\Gamma}\psi(t^{n+1})+\frac{1}{\delta}G'(\psi(t^{n}))
+\varepsilon\partial_{\mathbf{n}}\phi(t^{n+1}) \nonumber\\
&&\qquad \qquad +s_2(\psi(t^{n+1})-\psi(t^{n}))+R_{\Gamma}^{n+1}, \quad \mbox{on}\  \Gamma,
\end{eqnarray}
where
\begin{equation}\label{Rphi}
R_{\phi}^{n+1}=\frac{\phi(t^{n+1})-\phi(t^{n})}{\tau}-\phi_t(t^{n+1}),
\end{equation}
\begin{equation}\label{Rpsi}
R_{\psi}^{n+1}=\frac{\psi(t^{n+1})-\psi(t^{n})}{\tau}-\psi_t(t^{n+1}),
\end{equation}
\begin{equation}\label{Rmu}
R_{\mu}^{n+1}=\frac{1}{\varepsilon}F'(\phi(t^{n+1}))-\frac{1}{\varepsilon}F'(\phi(t^{n}))
-s_1(\phi(t^{n+1})-\phi(t^{n})),
\end{equation}
\begin{equation}\label{Rgamma}
R_{\Gamma}^{n+1}=\frac{1}{\delta}G'(\psi(t^{n+1}))-\frac{1}{\delta}G'(\psi(t^{n}))
-s_2(\psi(t^{n+1})-\psi(t^{n})).
\end{equation}

We assume that the exact solutions of the system \eqref{CHLW} possesses the following regularity:
\begin{equation}
(A):
\left\{
\begin{aligned}
&\phi,\phi_t,\phi_{tt}\in L^{\infty}(0,T; H^3(\Omega));\\
&\psi,\psi_t,\psi_{tt}\in L^{\infty}(0,T; H^3(\Gamma));\\
& \mu \in L^{\infty}(0,T; H^2(\Omega));\\
& \mu_{\Gamma} \in L^{\infty}(0,T; H^2(\Gamma)).
\end{aligned}
\right.
\end{equation}
From the Taylor expansion, it's easy to prove that

\begin{lemma}
Under the Assumption (A), the truncation errors satisfy
\begin{equation}
\begin{aligned}
&\|R_{\phi,\tau}\|_{l^{\infty}(H^1(\Omega))}+\|R_{\mu,\tau}\|_{l^{\infty}(H^1(\Omega))}\lesssim \tau,\\
&\|R_{\phi,\tau}\|_{l^{\infty}(L^2(\Omega))}+\|R_{\mu,\tau}\|_{l^{\infty}(L^2(\Omega))}\lesssim \tau,\\
&\|R_{\psi,\tau}\|_{l^{\infty}(H^1(\Gamma))}+\|R_{\Gamma,\tau}\|_{l^{\infty}(H^1(\Gamma))}\lesssim \tau,\\
&\|R_{\psi,\tau}\|_{l^{\infty}(L^2(\Gamma))}+\|R_{\Gamma,\tau}\|_{l^{\infty}(L^2(\Gamma))}\lesssim \tau.
\end{aligned}
\end{equation}
Here, the truncation errors are defined as Eq. \eqref{Rphi}-\eqref{Rgamma}.
And the corresponding sequences are denoted as $\{R_{\phi,\tau}\}$, $\{R_{\psi,\tau}\}$, $\{R_{\mu,\tau}\}$ and $\{R_{\Gamma,\tau}\}$ with $\tau$ the time step size.
Moreover, the discrete norm $\|\cdot\|_{l^{\infty}(\cdot)}$ is defined as Eq. \eqref{discretenorm}.

\end{lemma}

Thus we can establish the estimates for the stabilized scheme as follows.
\begin{theorem}
Provided that the exact solutions are sufficiently smooth or under the assumption (A), then for $0\leq m \leq [\frac{T}{\tau}]-1$,
the solution $(\phi^m,\psi^m)$ of the scheme \eqref{SIscheme1}-\eqref{SIscheme6} satisfy the following error estimates
\begin{equation}\label{finalestimate}
\begin{aligned}
&\|e_{\phi,\tau}\|_{l^{\infty}(H^1(\Omega))}+\|e_{\psi,\tau}\|_{l^{\infty}(H^1(\Gamma))}\lesssim\tau,\\
&\|e_{\phi,\tau}\|_{l^{\infty}(L^2(\Omega))}+\|e_{\psi,\tau}\|_{l^{\infty}(L^2(\Gamma))}\lesssim\tau.
\end{aligned}
\end{equation}
Here, the error functions are defined as
\begin{equation}
\begin{aligned}
&e_{\phi}^n=\phi(t^{n})-\phi^n,\qquad e_{\mu}^n=\mu(t^{n})-\mu^n,\\
&e_{\psi}^n=\psi(t^{n})-\psi^n, \qquad e_{\Gamma}^n=\mu_{\Gamma}(t^{n})-\mu_{\Gamma}^n,\\
&e_{\phi}^n|_{\Gamma}=e_{\psi}^n.
\end{aligned}
\end{equation}
The corresponding sequences of error functions are denoted as $e_{\phi,\tau}$, $e_{\psi,\tau}$, $e_{\mu,\tau}$ and $e_{\Gamma,\tau}$, and the discrete norm $\|\cdot\|_{l^{\infty}(\cdot)}$ is defined as Eq. \eqref{discretenorm}.
\end{theorem}

\begin{proof}
We use the mathematical induction to prove this theorem. When $m=0$, we have $e_{\phi}^0=e_{\psi}^0=\nabla e_{\phi}^0=\nabla_{\Gamma} e_{\psi}^0=0$. Obviously, \eqref{finalestimate} holds. Assuming that the error estimate holds for all $n\leq m$, we need to show that the error estimate holds for $e_{\phi}^{m+1}$ and $e_{\psi}^{m+1}$. For each $n\leq m$,
by subtracting \eqref{truncatePDE}-\eqref{truncatePDE6} from the corresponding scheme \eqref{SIscheme1}-\eqref{SIscheme6}, we derive the error equations as follows,
\begin{eqnarray}\label{erroreq}
&&\frac{1}{\tau}(e_{\phi}^{n+1}-e_{\phi}^{n})=\Delta e_{\mu}^{n+1}+R_{\phi}^{n+1}, \quad \mbox{in}\  \Omega,\\
\label{erroreq2}
&&e_{\mu}^{n+1}=-\varepsilon\Delta e_{\phi}^{n+1}+\frac{1}{\varepsilon}(F'(\phi(t^{n}))-F'(\phi^{n}))+s_1(e_{\phi}^{n+1}-e_{\phi}^{n})
+R_{\mu}^{n+1}, \quad \mbox{in}\ \Omega,\\
\label{erroreq3}
&&\partial_{\mathbf{n}}e_{\mu}^{n+1}=0, \quad \mbox{on}\  \Gamma,\\
\label{erroreq4}
&&e_{\phi}^{n+1}|_{\Gamma}=e_{\psi}^{n+1}, \quad \mbox{on}\  \Gamma,\\
\label{erroreq5}
&&\frac{1}{\tau}(e_{\psi}^{n+1}-e_{\psi}^{n})=\Delta_{\Gamma}e_{\Gamma}^{n+1}+R_{\psi}^{n+1}, \quad \mbox{on}\  \Gamma,\\
\label{erroreq6}
&&e_{\Gamma}^{n+1}=-\delta\kappa\Delta_{\Gamma}e_{\psi}^{n+1}+\frac{1}{\delta}(G'(\psi(t^{n}))-G'(\psi^{n}))
+\varepsilon\partial_{\mathbf{n}}e_{\phi}^{n+1} \nonumber\\
&&\qquad \quad +s_2(e_{\psi}^{n+1}-e_{\psi}^{n})+R_{\Gamma}^{n+1}, \quad \mbox{on}\  \Gamma.
\end{eqnarray}

By taking the $L^2$ inner product of \eqref{erroreq} with $\tau e_{\mu}^{n+1}$ in $\Omega$, we obtain
$$
(e_{\phi}^{n+1}-e_{\phi}^{n}, e_{\mu}^{n+1})_{\Omega}+\tau\|\nabla e_{\mu}^{n+1}\|_{\Omega}^2=\tau(R_{\phi}^{n+1},e_{\mu}^{n+1})_{\Omega}.
$$
By taking the $L^2$ inner product of \eqref{erroreq} with $\varepsilon\tau e_{\phi}^{n+1}$ in $\Omega$, we obtain
$$
\begin{aligned}
\frac{\varepsilon}2(\| e_{\phi}^{n+1}\|_{\Omega}^2-\|e_{\phi}^{n}\|_{\Omega}^2+\| e_{\phi}^{n+1}-e_{\phi}^{n}\|_{\Omega}^2)&=-\varepsilon\tau(\nabla e_{\mu}^{n+1},\nabla e_{\phi}^{n+1})_{\Omega}+\varepsilon\tau(R_{\phi}^{n+1},e_{\phi}^{n+1})_{\Omega},
\end{aligned}
$$
where the boundary terms vanish due to $\partial_{\mathbf{n}}e_{\mu}^{n+1}=0$.
By taking the $L^2$ inner product of \eqref{erroreq2} with $-(e_{\phi}^{n+1}-e_{\phi}^{n})$ in $\Omega$, we obtain
$$
\begin{aligned}
&-(e_{\mu}^{n+1}, e_{\phi}^{n+1}-e_{\phi}^{n})_{\Omega}+\frac{\varepsilon}2(\| \nabla e_{\phi}^{n+1}\|_{\Omega}^2-\|\nabla e_{\phi}^{n}\|_{\Omega}^2+\| \nabla e_{\phi}^{n+1}-\nabla e_{\phi}^{n}\|_{\Omega}^2)+s_1\|e_{\phi}^{n+1}-e_{\phi}^{n}\|_{\Omega}^2=\\
&\varepsilon(\partial_{\mathbf{n}}e_{\phi}^{n+1},
e_{\phi}^{n+1}-e_{\phi}^{n})_{\Gamma}-\frac{1}{\varepsilon}(F'(\phi(t^{n}))-F'(\phi^{n}), e_{\phi}^{n+1}-e_{\phi}^{n})_{\Omega}-(R_{\mu}^{n+1},e_{\phi}^{n+1}-e_{\phi}^{n})_{\Omega}.
\end{aligned}
$$

By combining the equations above, we derive
 \begin{equation}\label{errorin}
\begin{aligned}
&\frac{\varepsilon}2(\| \nabla e_{\phi}^{n+1}\|_{\Omega}^2-\|\nabla e_{\phi}^{n}\|_{\Omega}^2+\| \nabla e_{\phi}^{n+1}-\nabla e_{\phi}^{n}\|_{\Omega}^2)+s_1\|e_{\phi}^{n+1}-e_{\phi}^{n}\|_{\Omega}^2\\
&+\frac{\varepsilon}2(\| e_{\phi}^{n+1}\|_{\Omega}^2-\| e_{\phi}^{n}\|_{\Omega}^2+\| e_{\phi}^{n+1}- e_{\phi}^{n}\|_{\Omega}^2)+\tau\|\nabla e_{\mu}^{n+1}\|_{\Omega}^2\\
&=\tau(R_{\phi}^{n+1},e_{\mu}^{n+1})_{\Omega}-\varepsilon\tau(\nabla e_{\mu}^{n+1},\nabla e_{\phi}^{n+1})_{\Omega}+\varepsilon\tau(R_{\phi}^{n+1},e_{\phi}^{n+1})_{\Omega}\\
&-(R_{\mu}^{n+1},e_{\phi}^{n+1}-e_{\phi}^{n})_{\Omega}+\varepsilon(\partial_{\mathbf{n}}e_{\phi}^{n+1},
e_{\phi}^{n+1}-e_{\phi}^{n})_{\Gamma}-\frac{1}{\varepsilon}(F'(\phi(t^{n}))-F'(\phi^{n}), e_{\phi}^{n+1}-e_{\phi}^{n})_{\Omega}.
\end{aligned}
\end{equation}

For the boundary term, by taking the $L^2$ inner product of \eqref{erroreq5} with $\tau e_{\Gamma}^{n+1}$ on $\Gamma$, we obtain
$$
(e_{\psi}^{n+1}-e_{\psi}^{n}, e_{\Gamma}^{n+1})_{\Gamma}+\tau\|\nabla_{\Gamma} e_{\Gamma}^{n+1}\|_{\Gamma}^2=\tau(R_{\psi}^{n+1},e_{\Gamma}^{n+1})_{\Gamma}.
$$
By taking the $L^2$ inner product of \eqref{erroreq5} with $\delta\kappa\tau e_{\psi}^{n+1}$ on $\Gamma$, we obtain
$$
\begin{aligned}
\frac{\delta\kappa}2(\| e_{\psi}^{n+1}\|_{\Gamma}^2-\|e_{\psi}^{n}\|_{\Gamma}^2+\| e_{\psi}^{n+1}-e_{\psi}^{n}\|_{\Gamma}^2)&=-\delta\kappa\tau(\nabla_{\Gamma} e_{\Gamma}^{n+1},\nabla_{\Gamma} e_{\psi}^{n+1})_{\Gamma}+\delta\kappa\tau(R_{\psi}^{n+1},e_{\psi}^{n+1})_{\Gamma},
\end{aligned}
$$
where the boundary terms vanish due to $\Gamma$ is closed.
By taking the $L^2$ inner product of \eqref{erroreq6} with $-(e_{\psi}^{n+1}-e_{\psi}^{n})$ on $\Gamma$, we obtain
$$
\begin{aligned}
&-(e_{\Gamma}^{n+1}, e_{\psi}^{n+1}-e_{\psi}^{n})_{\Gamma}+\frac{\delta\kappa}2(\| \nabla_{\Gamma} e_{\psi}^{n+1}\|_{\Gamma}^2-\|\nabla_{\Gamma} e_{\psi}^{n}\|_{\Gamma}^2\\
&+\| \nabla_{\Gamma} e_{\psi}^{n+1}-\nabla_{\Gamma} e_{\psi}^{n}\|_{\Gamma}^2)
+s_2\|e_{\psi}^{n+1}-e_{\psi}^{n}\|_{\Gamma}^2\\
&=-\varepsilon(\partial_{\mathbf{n}}e_{\phi}^{n+1},
e_{\psi}^{n+1}-e_{\psi}^{n})_{\Gamma}-\frac{1}{\delta}(G'(\psi(t^{n}))-G'(\psi^{n}), e_{\psi}^{n+1}-e_{\psi}^{n})_{\Gamma}-(R_{\Gamma}^{n+1},e_{\psi}^{n+1}-e_{\psi}^{n})_{\Gamma}.
\end{aligned}
$$
By combining the equations above, we derive
 \begin{equation}\label{errorboun}
\begin{aligned}
&\frac{\delta\kappa}2(\| \nabla_{\Gamma} e_{\psi}^{n+1}\|_{\Gamma}^2-\|\nabla_{\Gamma} e_{\psi}^{n}\|_{\Gamma}^2+\| \nabla_{\Gamma} e_{\psi}^{n+1}-\nabla_{\Gamma} e_{\psi}^{n}\|_{\Gamma}^2)+s_2\|e_{\psi}^{n+1}-e_{\psi}^{n}\|_{\Gamma}^2\\
&+\frac{\delta\kappa}2(\| e_{\psi}^{n+1}\|_{\Gamma}^2-\| e_{\psi}^{n}\|_{\Gamma}^2+\| e_{\psi}^{n+1}- e_{\psi}^{n}\|_{\Gamma}^2)+\tau\|\nabla_{\Gamma} e_{\Gamma}^{n+1}\|_{\Gamma}^2\\
&=\tau(R_{\psi}^{n+1},e_{\Gamma}^{n+1})_{\Gamma}-\delta\kappa\tau(\nabla_{\Gamma} e_{\Gamma}^{n+1},\nabla_{\Gamma} e_{\psi}^{n+1})_{\Gamma}+\tau\delta\kappa(R_{\psi}^{n+1},e_{\psi}^{n+1})_{\Gamma}\\
&-(R_{\Gamma}^{n+1},e_{\psi}^{n+1}-e_{\psi}^{n})_{\Gamma}-\varepsilon(\partial_{\mathbf{n}}e_{\phi}^{n+1},
e_{\psi}^{n+1}-e_{\psi}^{n})_{\Gamma}-\frac{1}{\delta}(G'(\psi(t^{n}))-G'(\psi^{n}), e_{\psi}^{n+1}-e_{\psi}^{n})_{\Gamma}.
\end{aligned}
\end{equation}

By combining \eqref{errorin} and \eqref{errorboun} together, we derive
\begin{equation}\label{errorzong}
\begin{aligned}
&\frac{\varepsilon}2(\| \nabla e_{\phi}^{n+1}\|_{\Omega}^2-\|\nabla e_{\phi}^{n}\|_{\Omega}^2+\| \nabla e_{\phi}^{n+1}-\nabla e_{\phi}^{n}\|_{\Omega}^2)+\frac{\varepsilon}2(\| e_{\phi}^{n+1}\|_{\Omega}^2-\| e_{\phi}^{n}\|_{\Omega}^2+\| e_{\phi}^{n+1}- e_{\phi}^{n}\|_{\Omega}^2)\\
&+\frac{\delta\kappa}2(\| \nabla_{\Gamma} e_{\psi}^{n+1}\|_{\Gamma}^2-\|\nabla_{\Gamma} e_{\psi}^{n}\|_{\Gamma}^2+\| \nabla_{\Gamma} e_{\psi}^{n+1}-\nabla_{\Gamma} e_{\psi}^{n}\|_{\Gamma}^2)\\
&+\frac{\delta\kappa}2(\| e_{\psi}^{n+1}\|_{\Gamma}^2-\| e_{\psi}^{n}\|_{\Gamma}^2+\| e_{\psi}^{n+1}- e_{\psi}^{n}\|_{\Gamma}^2)
+s_1\|e_{\phi}^{n+1}-e_{\phi}^{n}\|_{\Omega}^2+s_2\|e_{\psi}^{n+1}-e_{\psi}^{n}\|_{\Gamma}^2\\
&+\tau\|\nabla e_{\mu}^{n+1}\|_{\Omega}^2+\tau\|\nabla_{\Gamma} e_{\Gamma}^{n+1}\|_{\Gamma}^2\\
&=\tau(R_{\phi}^{n+1},e_{\mu}^{n+1})_{\Omega}+\tau(R_{\psi}^{n+1},e_{\Gamma}^{n+1})_{\Gamma}\quad (:=\mbox{term }A_1)\\
&-(R_{\mu}^{n+1},e_{\phi}^{n+1}-e_{\phi}^{n})_{\Omega}-(R_{\Gamma}^{n+1},e_{\psi}^{n+1}-e_{\psi}^{n})_{\Gamma}
\quad (:=\mbox{term }A_2)\\
&-\varepsilon\tau(\nabla e_{\mu}^{n+1},\nabla e_{\phi}^{n+1})_{\Omega}-\delta\kappa\tau(\nabla_{\Gamma} e_{\Gamma}^{n+1},\nabla_{\Gamma} e_{\psi}^{n+1})_{\Gamma}\quad (:=\mbox{term }A_3)\\
&+\varepsilon\tau(R_{\phi}^{n+1},e_{\phi}^{n+1})_{\Omega}+\tau\delta\kappa(R_{\psi}^{n+1},e_{\psi}^{n+1})_{\Gamma}
\quad (:=\mbox{term }A_4)\\
&-\frac{1}{\varepsilon}(F'(\phi(t^{n}))-F'(\phi^{n}), e_{\phi}^{n+1}-e_{\phi}^{n})_{\Omega}-\frac{1}{\delta}(G'(\psi(t^{n}))-G'(\psi^{n}), e_{\psi}^{n+1}-e_{\psi}^{n})_{\Gamma}\\
&\quad (:=\mbox{term }A_5)
\end{aligned}
\end{equation}

For simplicity, we define $H^n=F'(\phi(t^{n}))-F'(\phi^{n})$. It can be rewritten as
\begin{equation}
H^n=e_{\phi}^{n}\int_0^1 F''(s\phi(t^{n})+(1-s)\phi^{n}) ds.
\end{equation}
We have $\|H^n\|_{\Omega}\lesssim\|e_{\phi}^{n}\|_{\Omega}$ since $F''$ is bounded.
By taking the gradient of $H^n$, we have
\begin{equation}
\begin{aligned}
\nabla H^n&=F''(\phi(t^{n}))\nabla \phi(t^{n})-F''(\phi^{n})\nabla\phi^{n}=(F''(\phi(t^{n}))-F''(\phi^{n}))\nabla \phi(t^{n})+F''(\phi^{n})\nabla e_{\phi}^{n}
\end{aligned}
\end{equation}
Since $F''$ is bounded and Lipschitz and assumption (A), we have
\begin{equation}
\begin{aligned}
\|\nabla H^n\|_{\Omega}&\lesssim\|e_{\phi}^{n}\|_{\Omega}\|\phi(t^n)\|_{H^3(\Omega)}+\|\nabla e_{\phi}^{n}\|_{\Omega}\\
& \lesssim \|e_{\phi}^{n}\|_{\Omega}+\|\nabla e_{\phi}^{n}\|_{\Omega}.
\end{aligned}
\end{equation}
Similarly, we define $\tilde{H}^n=G'(\psi(t^{n}))-G'(\psi^{n})$ for simplicity. Since $G''$ is bounded and Lipschitz and assumption (A), we have
\begin{equation}
\begin{aligned}
&\|\tilde{H}^n\|_{\Gamma}\lesssim\|e_{\psi}^{n}\|_{\Gamma},\\
&\|\nabla_{\Gamma} \tilde{H}^n\|_{\Gamma}\lesssim\|e_{\psi}^{n}\|_{\Gamma}+\|\nabla_{\Gamma} e_{\psi}^{n}\|_{\Gamma}.
\end{aligned}
\end{equation}

For the term $A_1$, we have
\begin{equation}\label{A1}
\begin{aligned}
&\tau(R_{\phi}^{n+1},e_{\mu}^{n+1})_{\Omega}+\tau(R_{\psi}^{n+1},e_{\Gamma}^{n+1})_{\Gamma}\\
&=\tau(R_{\phi}^{n+1}, -\varepsilon\Delta e_{\phi}^{n+1}+\frac{1}{\varepsilon}H^n
+s_1(e_{\phi}^{n+1}-e_{\phi}^{n})+R_{\mu}^{n+1})_{\Omega}\\
&+\tau(R_{\psi}^{n+1}, -\delta\kappa\Delta_{\Gamma}e_{\psi}^{n+1}+\frac{1}{\delta}\tilde{H}^n
+\varepsilon\partial_{\mathbf{n}}e_{\phi}^{n+1}
+s_2(e_{\psi}^{n+1}-e_{\psi}^{n})+R_{\Gamma}^{n+1})_{\Gamma}\\
&=\varepsilon\tau(\nabla R_{\phi}^{n+1}, \nabla e_{\phi}^{n+1})_{\Omega}+\frac{\tau}{\varepsilon}(H^n, R_{\phi}^{n+1})_{\Omega}+
s_1\tau(R_{\phi}^{n+1}, e_{\phi}^{n+1}-e_{\phi}^{n})_{\Omega}\\
&+\tau(R_{\phi}^{n+1}, R_{\mu}^{n+1})_{\Omega}+\tau\delta\kappa(\nabla_{\Gamma}R_{\psi}^{n+1}, \nabla_{\Gamma} e_{\psi}^{n+1})_{\Gamma}+\frac{\tau}{\delta}(\tilde{H}^n, R_{\psi}^{n+1})_{\Gamma}\\
&+s_2\tau(R_{\psi}^{n+1}, e_{\psi}^{n+1}-e_{\psi}^{n})_{\Gamma}+\tau(R_{\psi}^{n+1},R_{\Gamma}^{n+1})_{\Gamma}\\
&\leq \varepsilon\tau \|\nabla R_{\phi}^{n+1}\|_{\Omega}\|\nabla e_{\phi}^{n+1}\|_{\Omega}+\frac{\tau}{\varepsilon}\|H^n\|_{\Omega}\| R_{\phi}^{n+1}\|_{\Omega}+s_1\tau\|R_{\phi}^{n+1}\|_{\Omega}\| e_{\phi}^{n+1}-e_{\phi}^{n}\|_{\Omega}\\
&+\tau\|R_{\phi}^{n+1}\|_{\Omega}\|R_{\mu}^{n+1}\|_{\Omega}+\tau\delta\kappa\|
\nabla_{\Gamma}R_{\psi}^{n+1}\|_{\Gamma}\|\nabla_{\Gamma} e_{\psi}^{n+1}\|_{\Gamma}+
\frac{\tau}{\delta}\|\tilde{H}^n\|_{\Gamma}\|R_{\psi}^{n+1}\|_{\Gamma}\\
&+s_2\tau\|R_{\psi}^{n+1}\|_{\Gamma}\|e_{\psi}^{n+1}-e_{\psi}^{n}\|_{\Gamma}+
\tau\|R_{\psi}^{n+1}\|_{\Gamma}\|R_{\Gamma}^{n+1}\|_{\Gamma}\\
&\leq \frac{\varepsilon\tau}2 \|\nabla R_{\phi}^{n+1}\|_{\Omega}^2+\frac{\varepsilon\tau}2\|\nabla e_{\phi}^{n+1}\|_{\Omega}^2+\frac{\tau}{2\varepsilon}\|H^n\|_{\Omega}^2+\frac{\tau}{2\varepsilon}
\|R_{\phi}^{n+1}\|_{\Omega}^2\\
&+\frac{s_1\tau}{2}\|R_{\phi}^{n+1}\|_{\Omega}^2+\frac{s_1\tau}{2}\| e_{\phi}^{n+1}-e_{\phi}^{n}\|_{\Omega}^2+\frac{\tau}{2}\|R_{\phi}^{n+1}\|_{\Omega}^2+
\frac{\tau}{2}\|R_{\mu}^{n+1}\|_{\Omega}^2\\
&+\frac{\tau\delta\kappa}{2}\|
\nabla_{\Gamma}R_{\psi}^{n+1}\|_{\Gamma}^2+\frac{\tau\delta\kappa}{2}\|\nabla_{\Gamma} e_{\psi}^{n+1}\|_{\Gamma}^2+\frac{\tau}{2\delta}\|\tilde{H}^n\|_{\Gamma}^2+\frac{\tau}{2\delta}
\|R_{\psi}^{n+1}\|_{\Gamma}^2\\
&+\frac{s_2\tau}{2}\|R_{\psi}^{n+1}\|_{\Gamma}^2+\frac{s_2\tau}{2}\|e_{\psi}^{n+1}-e_{\psi}^{n}\|_{\Gamma}^2
+\frac{\tau}{2}\|R_{\psi}^{n+1}\|_{\Gamma}^2+\frac{\tau}{2}\|R_{\Gamma}^{n+1}\|_{\Gamma}^2\\
&\leq C_1 \tau^3+\frac{\varepsilon\tau}2\|\nabla e_{\phi}^{n+1}\|_{\Omega}^2+C_2 \tau \| e_{\phi}^{n}\|_{\Omega}^2+\frac{s_1\tau}{2}\| e_{\phi}^{n+1}-e_{\phi}^{n}\|_{\Omega}^2\\
&+\frac{\tau\delta\kappa}{2}\|\nabla_{\Gamma} e_{\psi}^{n+1}\|_{\Gamma}^2+C_3 \tau\| e_{\psi}^{n}\|_{\Gamma}^2+\frac{s_2\tau}{2}\| e_{\psi}^{n+1}-e_{\psi}^{n}\|_{\Gamma}^2,
\end{aligned}
\end{equation}
where $C_i$ ($i=1,2,3$) are constants independent of $\tau$. Here, we use the estimates for $H^n$ and $\tilde{H}^n$ and the truncation terms $R_{\phi}^{n+1}$, $R_{\psi}^{n+1}$, $R_{\mu}^{n+1}$ and $R_{\Gamma}^{n+1}$.

For the terms in $A_2$, we have
\begin{equation}\label{A2_1}
\begin{aligned}
&-(R_{\mu}^{n+1},e_{\phi}^{n+1}-e_{\phi}^{n})_{\Omega}
=-\tau(R_{\mu}^{n+1},\frac{e_{\phi}^{n+1}-e_{\phi}^{n}}{\tau})_{\Omega}\\
&=-\tau(R_{\mu}^{n+1}, \Delta e_{\mu}^{n+1}+R_{\phi}^{n+1})_{\Omega}
=\tau(\nabla R_{\mu}^{n+1}, \nabla e_{\mu}^{n+1})_{\Omega}-\tau(R_{\mu}^{n+1},R_{\phi}^{n+1})_{\Omega}\\
&\leq \tau \|\nabla R_{\mu}^{n+1}\|_{\Omega}\|\nabla e_{\mu}^{n+1}\|_{\Omega}+\tau \|R_{\mu}^{n+1}\|_{\Omega}\|R_{\phi}^{n+1}\|_{\Omega}\\
&\leq 2\tau\|\nabla R_{\mu}^{n+1}\|_{\Omega}^2+\frac{\tau}8\|\nabla e_{\mu}^{n+1}\|_{\Omega}^2+\frac{\tau}2\| R_{\mu}^{n+1}\|_{\Omega}^2+\frac{\tau}2\| R_{\phi}^{n+1}\|_{\Omega}^2\\
&\leq C_4 \tau^3+\frac{\tau}8\|\nabla e_{\mu}^{n+1}\|_{\Omega}^2,
\end{aligned}
\end{equation}
and
\begin{equation}\label{A2_2}
\begin{aligned}
&-(R_{\Gamma}^{n+1},e_{\psi}^{n+1}-e_{\psi}^{n})_{\Gamma}
=-\tau(R_{\Gamma}^{n+1},\frac{e_{\psi}^{n+1}-e_{\psi}^{n}}{\tau})_{\Gamma}\\
&=-\tau(R_{\Gamma}^{n+1}, \Delta_{\Gamma} e_{\Gamma}^{n+1}+R_{\psi}^{n+1})_{\Gamma}
=\tau(\nabla_{\Gamma} R_{\Gamma}^{n+1}, \nabla_{\Gamma} e_{\Gamma}^{n+1})_{\Gamma}-\tau(R_{\Gamma}^{n+1},R_{\psi}^{n+1})_{\Gamma}\\
&\leq 2\tau\|\nabla_{\Gamma} R_{\Gamma}^{n+1}\|_{\Gamma}^2+\frac{\tau}8\|\nabla_{\Gamma} e_{\Gamma}^{n+1}\|_{\Gamma}^2+\frac{\tau}2\| R_{\Gamma}^{n+1}\|_{\Gamma}^2+\frac{\tau}2\| R_{\psi}^{n+1}\|_{\Gamma}^2\\
&\leq C_5 \tau^3+\frac{\tau}8\|\nabla_{\Gamma} e_{\Gamma}^{n+1}\|_{\Gamma}^2,
\end{aligned}
\end{equation}
where $C_i$ ($i=4,5$) are constants independent of $\tau$. Here, we use the estimates for the truncation terms $R_{\phi}^{n+1}$, $R_{\psi}^{n+1}$, $R_{\mu}^{n+1}$ and $R_{\Gamma}^{n+1}$.

We estimate $A_3$ as follows
\begin{equation}\label{A3}
\begin{aligned}
&-\varepsilon\tau(\nabla e_{\mu}^{n+1},\nabla e_{\phi}^{n+1})_{\Omega}-\delta\kappa\tau(\nabla_{\Gamma} e_{\Gamma}^{n+1},\nabla_{\Gamma} e_{\psi}^{n+1})_{\Gamma}\\
&\leq \varepsilon\tau\|\nabla e_{\mu}^{n+1}\|_{\Omega}\|\nabla e_{\phi}^{n+1}\|_{\Omega}+
\delta\kappa\tau \|\nabla_{\Gamma} e_{\Gamma}^{n+1}\|_{\Gamma}\|\nabla_{\Gamma} e_{\psi}^{n+1}\|_{\Gamma}\\
&\leq 2\varepsilon^2\tau\|\nabla e_{\phi}^{n+1}\|_{\Omega}^2+\frac{\tau}8 \|\nabla e_{\mu}^{n+1}\|_{\Omega}^2+2\delta^2\kappa^2\tau\|\nabla_{\Gamma} e_{\psi}^{n+1}\|_{\Gamma}^2+\frac{\tau}8\|\nabla_{\Gamma} e_{\Gamma}^{n+1}\|_{\Gamma}^2.
\end{aligned}
\end{equation}

For the term $A_4$, we have
\begin{equation}\label{A4}
\begin{aligned}
&\varepsilon\tau(R_{\phi}^{n+1},e_{\phi}^{n+1})_{\Omega}+\tau\delta\kappa(R_{\psi}^{n+1},e_{\psi}^{n+1})_{\Gamma}\\
&\leq\varepsilon\tau\|R_{\phi}^{n+1}\|_{\Omega}\|e_{\phi}^{n+1}\|_{\Omega}+\tau\delta\kappa
\|R_{\psi}^{n+1}\|_{\Gamma}\|e_{\psi}^{n+1}\|_{\Gamma}\\
&\leq C_6 \varepsilon\tau^3+\frac{\varepsilon\tau}{2}\|e_{\phi}^{n+1}\|_{\Omega}^2
+\frac{\delta\kappa\tau^3}{2}+\frac{\delta\kappa\tau}{2}\|e_{\psi}^{n+1}\|_{\Gamma}^2.
\end{aligned}
\end{equation}
Here, $C_6$ is a constant independent of $\tau$ and we use the estimates for the truncation terms $R_{\phi}^{n+1}$ and $R_{\psi}^{n+1}$.

For the terms in $A_5$, we have
\begin{equation}\label{A5_1}
\begin{aligned}
&-\frac{1}{\varepsilon}(F'(\phi(t^{n}))-F'(\phi^{n}), e_{\phi}^{n+1}-e_{\phi}^{n})_{\Omega}\\
&=-\frac{\tau}{\varepsilon}(H^n, \frac{e_{\phi}^{n+1}-e_{\phi}^{n}}{\tau})_{\Omega}
=-\frac{\tau}{\varepsilon}(H^n, \Delta e_{\mu}^{n+1}+R_{\phi}^{n+1})_{\Omega}\\
&=\frac{\tau}{\varepsilon}(\nabla H^n, \nabla e_{\mu}^{n+1})_{\Omega}-\frac{\tau}{\varepsilon}(H^n, R_{\phi}^{n+1})_{\Omega}\\
&\leq\frac{\tau}{\varepsilon}\|\nabla H^n\|_{\Omega}\|\nabla e_{\mu}^{n+1}\|_{\Omega}+
\frac{\tau}{\varepsilon}\|H^n\|_{\Omega}\|R_{\phi}^{n+1}\|_{\Omega}\\
&\leq C_7 \tau(\|e_{\phi}^{n}\|_{\Omega}+\|\nabla e_{\phi}^{n}\|_{\Omega})\|\nabla e_{\mu}^{n+1}\|_{\Omega}+C_8\tau\|e_{\phi}^{n}\|_{\Omega}\|R_{\phi}^{n+1}\|_{\Omega}\\
&\leq C_9\tau \|e_{\phi}^{n}\|_{\Omega}^2+2C_7^2\tau \|\nabla e_{\phi}^{n}\|_{\Omega}^2
+\frac{\tau}{4} \|\nabla e_{\mu}^{n+1}\|_{\Omega}^2 +C_{10} \tau^3,
\end{aligned}
\end{equation}
where $C_i$ ($i=7,8,9, 10$) are constants independent of $\tau$ and $C_9=2C_7^2+C_8/2$. Here, we use the estimates for $H^n$, $\nabla H^n$ and $R_{\phi}^{n+1}$.

\begin{equation}\label{A5_2}
\begin{aligned}
&-\frac{1}{\delta}(G'(\psi(t^{n}))-G'(\psi^{n}), e_{\psi}^{n+1}-e_{\psi}^{n})_{\Gamma}\\
&=-\frac{\tau}{\delta}(\tilde{H}^n, \frac{e_{\psi}^{n+1}-e_{\psi}^{n}}{\tau})_{\Gamma}
=-\frac{\tau}{\delta}(\tilde{H}^n, \Delta_{\Gamma} e_{\Gamma}^{n+1}+R_{\psi}^{n+1})_{\Gamma}\\
&=\frac{\tau}{\delta}(\nabla_{\Gamma}\tilde{H}^n,\nabla_{\Gamma} e_{\Gamma}^{n+1})_{\Gamma}-\frac{\tau}{\delta}(\tilde{H}^n,R_{\psi}^{n+1})_{\Gamma}\\
&\leq \frac{\tau}{\delta} \|\nabla_{\Gamma}\tilde{H}^n\|_{\Gamma}\|\nabla_{\Gamma}e_{\Gamma}^{n+1}\|_{\Gamma}+
\frac{\tau}{\delta}\| \tilde{H}^n\|_{\Gamma}\|R_{\psi}^{n+1}\|_{\Gamma}\\
&\leq C_{11} \tau( \|e_{\psi}^{n}\|_{\Gamma}+\|\nabla_{\Gamma} e_{\psi}^{n}\|_{\Gamma})\|\nabla_{\Gamma}e_{\Gamma}^{n+1}\|_{\Gamma}+C_{12}\tau
\|e_{\psi}^{n}\|_{\Gamma}\|R_{\psi}^{n+1}\|_{\Gamma}\\
&\leq C_{13}\tau\|e_{\psi}^{n}\|_{\Gamma}^2+2C_{11}^2\tau\|\nabla_{\Gamma} e_{\psi}^{n}\|_{\Gamma}^2+\frac{\tau}{4}\|\nabla_{\Gamma}e_{\Gamma}^{n+1}
\|_{\Gamma}^2+C_{14} \tau^3
\end{aligned}
\end{equation}
where $C_i$ ($i=11, 12, 13, 14$) are constants independent of $\tau$ and $C_{13}=2C_{11}^2+C_{12}/2$. Here, we use the estimates for $\tilde{H}^n$, $\nabla_{\Gamma} \tilde{H}^n$ and $R_{\psi}^{n+1}$.

Combine \eqref{errorzong} with \eqref{A1}, \eqref{A2_1}, \eqref{A2_2}, \eqref{A3}, \eqref{A4}, \eqref{A5_1} and \eqref{A5_2}, we derive
\begin{equation}\label{errorzong2}
\begin{aligned}
&\frac{\varepsilon}2(\| \nabla e_{\phi}^{n+1}\|_{\Omega}^2-\|\nabla e_{\phi}^{n}\|_{\Omega}^2+\| \nabla e_{\phi}^{n+1}-\nabla e_{\phi}^{n}\|_{\Omega}^2)
+\frac{\varepsilon}2(\| e_{\phi}^{n+1}\|_{\Omega}^2-\| e_{\phi}^{n}\|_{\Omega}^2+\| e_{\phi}^{n+1}- e_{\phi}^{n}\|_{\Omega}^2)\\
&+\frac{\delta\kappa}2(\| \nabla_{\Gamma} e_{\psi}^{n+1}\|_{\Gamma}^2-\|\nabla_{\Gamma} e_{\psi}^{n}\|_{\Gamma}^2+\| \nabla_{\Gamma} e_{\psi}^{n+1}-\nabla_{\Gamma} e_{\psi}^{n}\|_{\Gamma}^2)\\
&+\frac{\delta\kappa}2(\| e_{\psi}^{n+1}\|_{\Gamma}^2-\| e_{\psi}^{n}\|_{\Gamma}^2+\| e_{\psi}^{n+1}- e_{\psi}^{n}\|_{\Gamma}^2)
+s_1\|e_{\phi}^{n+1}-e_{\phi}^{n}\|_{\Omega}^2+s_2\|e_{\psi}^{n+1}-e_{\psi}^{n}\|_{\Gamma}^2\\
&+\frac{\tau}2(\|\nabla e_{\mu}^{n+1}\|_{\Omega}^2+\|\nabla_{\Gamma} e_{\Gamma}^{n+1}\|_{\Gamma}^2)\\
&\lesssim \tau^3+\tau(\| \nabla e_{\phi}^{n+1}\|_{\Omega}^2+\| \nabla e_{\phi}^{n}\|_{\Omega}^2+\|e_{\phi}^{n+1}\|_{\Omega}^2+\|e_{\phi}^{n}\|_{\Omega}^2+\| e_{\phi}^{n+1}- e_{\phi}^{n}\|_{\Omega}^2\\
&+\| \nabla_{\Gamma} e_{\psi}^{n+1}\|_{\Gamma}^2+\| \nabla_{\Gamma} e_{\psi}^{n}\|_{\Gamma}^2+\|e_{\psi}^{n+1}\|_{\Gamma}^2+\|e_{\psi}^{n}\|_{\Gamma}^2+\| e_{\psi}^{n+1}- e_{\psi}^{n}\|_{\Gamma}^2).
\end{aligned}
\end{equation}

Summing \eqref{errorzong2} together for $n=0$ to $m$, we derive
\begin{equation}\label{sum}
\begin{aligned}
&\frac{\varepsilon}2\| \nabla e_{\phi}^{m+1}\|_{\Omega}^2+\frac{\varepsilon}2\| e_{\phi}^{m+1}\|_{\Omega}^2+\frac{\delta\kappa}2\| \nabla_{\Gamma} e_{\psi}^{m+1}\|_{\Gamma}^2+\frac{\delta\kappa}2\| e_{\psi}^{m+1}\|_{\Gamma}^2\\
&+\sum_{n=0}^m \bigg{(} \frac{\varepsilon}2\| \nabla e_{\phi}^{n+1}-\nabla e_{\phi}^{n}\|_{\Omega}^2 +(\frac{\varepsilon}2+s_1)\| e_{\phi}^{n+1}- e_{\phi}^{n}\|_{\Omega}^2\\
&+\frac{\delta\kappa}2 \| \nabla_{\Gamma} e_{\psi}^{n+1}-\nabla_{\Gamma} e_{\psi}^{n}\|_{\Gamma}^2 +(\frac{\delta\kappa}2+s_2)\| e_{\psi}^{n+1}- e_{\psi}^{n}\|_{\Gamma}^2\\
&+ \frac{\tau}2(\|\nabla e_{\mu}^{n+1}\|_{\Omega}^2+\|\nabla_{\Gamma} e_{\Gamma}^{n+1}\|_{\Gamma}^2)  \bigg{)}\\
&\leq \tilde{C}(m+1)\tau^3+\tilde{C}\tau\sum_{n=0}^m\bigg{(}\| \nabla e_{\phi}^{n+1}\|_{\Omega}^2+\|e_{\phi}^{n+1}\|_{\Omega}^2+\| e_{\phi}^{n+1}- e_{\phi}^{n}\|_{\Omega}^2\\
&+\|\nabla_{\Gamma} e_{\psi}^{n+1}\|_{\Gamma}^2+\|e_{\psi}^{n+1}\|_{\Gamma}^2+\| e_{\psi}^{n+1}- e_{\psi}^{n}\|_{\Gamma}^2\bigg{)},
\end{aligned}
\end{equation}
where we use $e_{\phi}^0=e_{\psi}^0=\nabla e_{\phi}^0=\nabla_{\Gamma} e_{\psi}^0=0$.

Denote
\begin{equation}\label{Idef}
\begin{aligned}
I_m&=\frac{\varepsilon}2\| \nabla e_{\phi}^{m+1}\|_{\Omega}^2+\frac{\varepsilon}2\| e_{\phi}^{m+1}\|_{\Omega}^2+\frac{\delta\kappa}2\| \nabla_{\Gamma} e_{\psi}^{m+1}\|_{\Gamma}^2+\frac{\delta\kappa}2\| e_{\psi}^{m+1}\|_{\Gamma}^2\\
&+(\frac{\varepsilon}2+s_1)\| e_{\phi}^{m+1}- e_{\phi}^{m}\|_{\Omega}^2+(\frac{\delta\kappa}2+s_2)\| e_{\psi}^{m+1}- e_{\psi}^{m}\|_{\Gamma}^2
\end{aligned}
\end{equation}
and
\begin{equation}\label{Sdef}
\begin{aligned}
S_m&=\sum_{n=0}^m \bigg{(} \frac{\varepsilon}2\| \nabla e_{\phi}^{n+1}-\nabla e_{\phi}^{n}\|_{\Omega}^2+\frac{\delta\kappa}2 \| \nabla_{\Gamma} e_{\psi}^{n+1}-\nabla_{\Gamma} e_{\psi}^{n}\|_{\Gamma}^2\\
&+ \frac{\tau}2(\|\nabla e_{\mu}^{n+1}\|_{\Omega}^2+\|\nabla_{\Gamma} e_{\Gamma}^{n+1}\|_{\Gamma}^2)  \bigg{)}.
\end{aligned}
\end{equation}
Then we have
\begin{equation}\label{gronwall}
I_m+S_m\lesssim\tau^2+\tau\sum_{n=0}^m I_n.
\end{equation}

According to the discrete Gronwall's inequality, there exists constants $\tilde{c}_0$ and $C_0$, such that
\begin{equation}\label{gronwall2}
I_m+S_m\leq\tilde{c}_0\tau^2,
\end{equation}
where $\tilde{c}_0$ is independent of $\tau$ and $\tau \leq C_0$.
And thus the error estimate \eqref{finalestimate} holds for $e_{\phi}^{m+1}$ and $e_{\psi}^{m+1}$.
\end{proof}

\begin{remark}
When the surface diffusion is absent, namely, when $\kappa=0$, the scheme \eqref{SIscheme1}-\eqref{SIscheme6} is also valid and the energy stability and error estimates also hold. In this case, we only need to let $\kappa=0$ in Eq. \eqref{SIscheme6} to get the corresponding numerical scheme. Moreover, the proof for the stability and error estimates are similar to those mentioned above. The only difference is to let $\kappa=0$ in the proof. Thus, we omit the details here and leave it to interested readers.


\end{remark}

\section{Numerical simulations\label{s5}}

In this section, we present numerical experiments of the Liu-Wu model by implementing
the developed scheme \eqref{SIscheme1}-\eqref{SIscheme6}. The numerical examples include the comparison with the numerical results in \cite{knopf2019} and \cite{Garcke2020}, accuracy tests with respect to the time step size, and the simulation of the shape deformation of a square shaped droplet.

The discrete energy and mass are defined as
\begin{equation}
\begin{aligned}
E(\phi^{n},\psi^{n})&=E_{bulk}(\phi^{n})+E_{surf}(\psi^{n})\\
&=\int_{\Omega}\frac{1}{\varepsilon} F(\phi^{n})+\frac{\varepsilon}{2} |\nabla\phi^{n}|^2\mbox{d}x+\int_{\Gamma} \frac{1}{\delta} G(\psi^{n})+\frac{\delta\kappa}{2}|\nabla_{\Gamma}\psi^{n}|^2\mbox{d}S,
\end{aligned}
\end{equation}
\begin{equation}
\begin{aligned}
M(\phi^{n},\psi^{n})&=M^{bulk}(\phi^{n})+M^{surf}(\psi^{n})\\
&=\int_{\Omega}\phi^{n}\mbox{d}x+\int_{\Gamma} \psi^{n}\mbox{d}S.
\end{aligned}
\end{equation}
The time evolutions of energy and mass are plotted in this section to validate the stability of the numerical scheme and the conservation of mass.

In this section, we present the numerical simulations in two dimensions.
For the spatial operators, we use the second-order central finite difference
method to discretize them over a uniform spatial grid.

\subsection{Comparison with former work}

We reconstruct the numerical experiments in \cite{knopf2019} and \cite{Garcke2020}
to validate the accuracy and robustness of our scheme.
\subsubsection{Comparison with numerical experiments in \cite{knopf2019}}

Firstly, we consider the initial condition
\begin{equation}\label{lam4initial}
\phi_0(x,y)=\left\{\begin{aligned}
&1 \qquad \mbox{if} \ x>1/2,\\
&-1\quad \mbox{if} \ x\leq1/2,
\end{aligned}
\right.
\end{equation}
where $\mathbf{x}=(x,y)\in[0,1]^2$, which is plotted in Fig. \ref{lam34initial}. The time step size $\tau=10^{-5}$ and the spatial step size $h=0.01$.
The parameters are set the same as those in Section 7.2.1 of \cite{knopf2019}:
$\varepsilon=1$, $\delta=0.1$, $\kappa=1$. $F$ and $G$ are chosen to be the classical double-well potential \eqref{classicalF}. We set $s_1=1$, $s_2=10$ to make sure that the scheme \eqref{SIscheme1}-\eqref{SIscheme6} is energy stable.

We observe that the numerical solutions are almost constant in the orthogonal direction. Thus, the projection of the numerical solution on the line $y=1/2$ after 200 time steps is plotted in Fig. \ref{lam4}, indicating the dissipation in the bulk. It is consistent with the results in \cite{knopf2019}. The energy and mass evolution with respect to time are also shown in Fig. \ref{lam4}, revealing the energy stability and the mass conservation both in the bulk and on the boundary.
\begin{figure}
\centering
\includegraphics[scale=0.25]{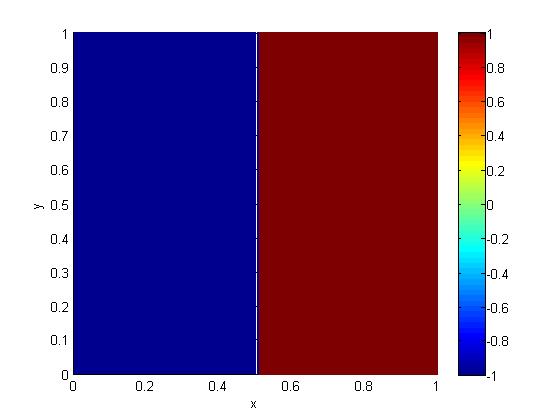}
\includegraphics[scale=0.25]{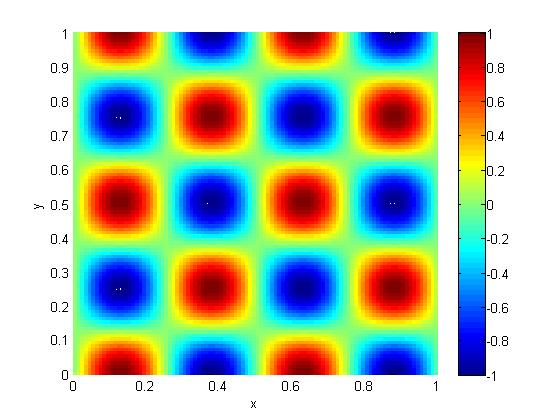}
\caption{The initial data of \eqref{lam4initial} and \eqref{lamsincosinitial}.}
\label{lam34initial}
\end{figure}
\begin{figure}
\centering
\includegraphics[scale=0.22]{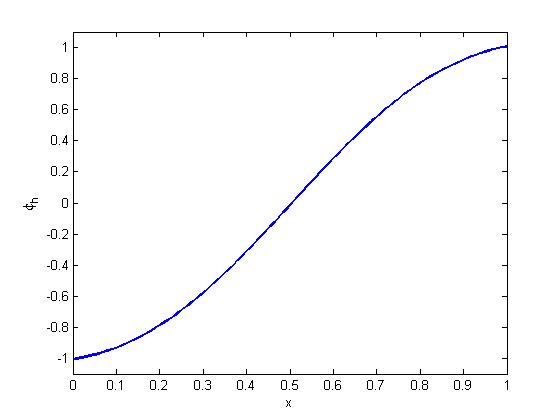}
\includegraphics[scale=0.22]{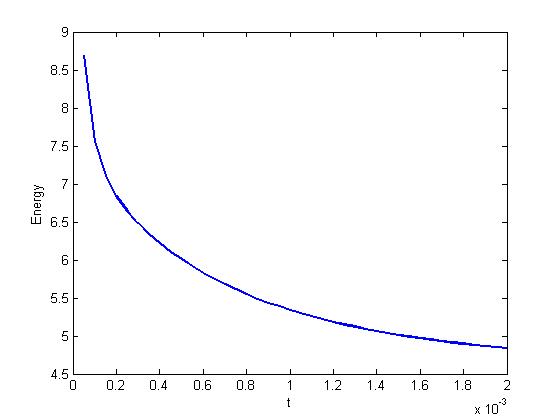}
\includegraphics[scale=0.22]{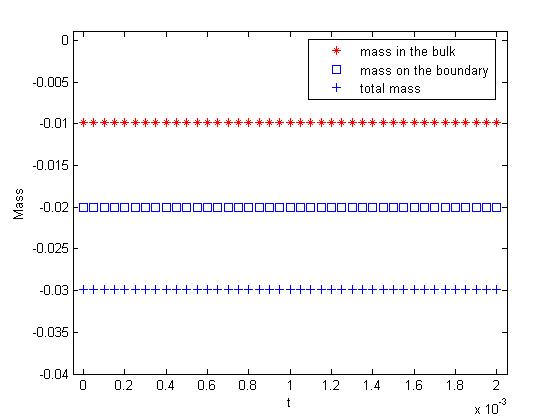}
\caption{For the initial data \eqref{lam4initial}, the projection of the numerical solution on the line $y=1/2$ after 200 time steps (left), energy evolution (middle) and mass evolution (right).}
\label{lam4}
\end{figure}

Next, we consider the initial data
\begin{equation}\label{lamsincosinitial}
\phi_0(x,y)=\sin(4\pi x)\cos(4\pi y),
\end{equation}
where $\mathbf{x}=(x,y)\in[0,1]^2$, which is also plotted in Fig. \ref{lam34initial}. The time step size $\tau=10^{-5}$ and the spatial step size $h=0.01$.

The parameters are set the same as those in Section 7.2.2 of \cite{knopf2019}:
$\varepsilon=\delta=0.02$, $\kappa=1$. $F$ and $G$ are chosen to be the classical double-well potential \eqref{classicalF}. We set $s_1=s_2=50$ to make sure that the scheme \eqref{SIscheme1}-\eqref{SIscheme6} is energy stable.
The numerical solution after 100 time steps is plotted in Fig. \ref{lamsincos}, which is qualitatively consistent with the numerical results in \cite{knopf2019}.
The energy and mass evolutions are also plotted in Fig. \ref{lamsincos}, showing the energy stability of the numerical scheme and the mass conservation both in the bulk and on the boundary.

\begin{figure}
\centering
\includegraphics[scale=0.22]{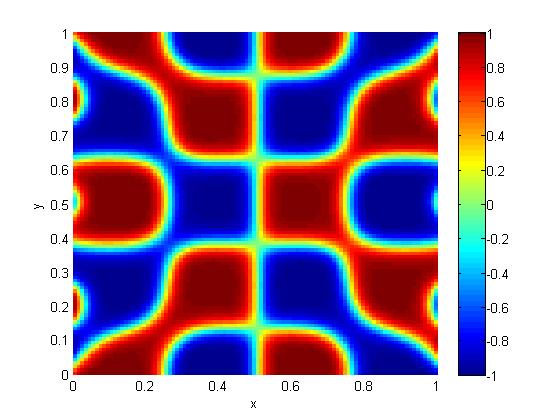}
\includegraphics[scale=0.22]{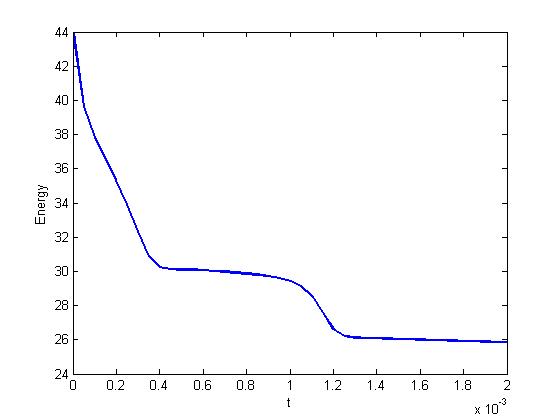}
\includegraphics[scale=0.22]{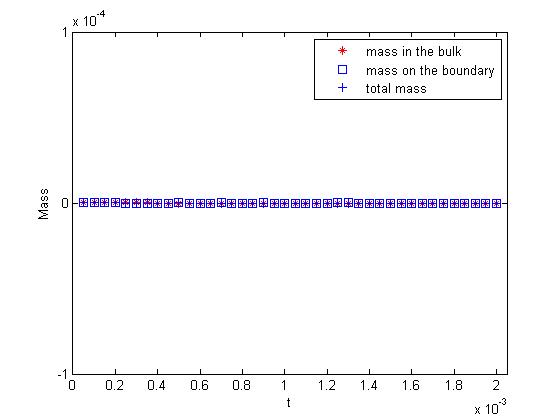}
\caption{ Numerical results of Liu-Wu model with the initial data \eqref{lamsincosinitial} after 100 time steps (left), energy evolution (middle) and mass evolution (right).}
\label{lamsincos}
\end{figure}

\subsubsection{Comparison with numerical experiments in \cite{Garcke2020}}

Firstly, the parameters are the same as those of the first numerical simulation in \cite{Garcke2020}: $\varepsilon=\delta=0.02$, $\kappa=0.02$. $\Omega$ is the unit square and the spatial step size $h=0.01$.
$F$ and $G$ are chosen to be the classical double-well potential \eqref{classicalF}, and we set $s_1=s_2=100$ to make sure that the scheme \eqref{SIscheme1}-\eqref{SIscheme6} is energy stable.
The time step size is set as $\tau=8\times10^{-6}$. The initial data $\phi_0$ is set to be zero in the bulk and set to be one on the boundary.

The energy evolution and the mass evolutions are plotted in Fig. \ref{lam12energy} and \ref{lam1mass}. The energy decreases with respect to time, indicating the stability of the scheme.
The masses in the bulk and on the boundary are conserved respectively, which is consistent with the properties of Liu-Wu model.
The numerical solutions after 5, 15, 80, 200, 500 and 2500 time steps are shown in Fig. \ref{01}. Due to the mass conservation on the boundary, the numerical solution remain to be 1 on the boundary. A wave-like structure arises during the phase separation in the bulk, and  ultimately, a circle of the phase of value -1 appears at the center of $\Omega$. The numerical results are consistent with the former work \cite{Garcke2020}.

\begin{figure}
\centering
\includegraphics[scale=0.33]{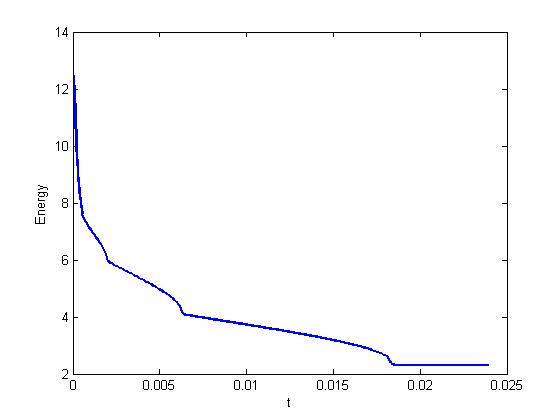}
\includegraphics[scale=0.33]{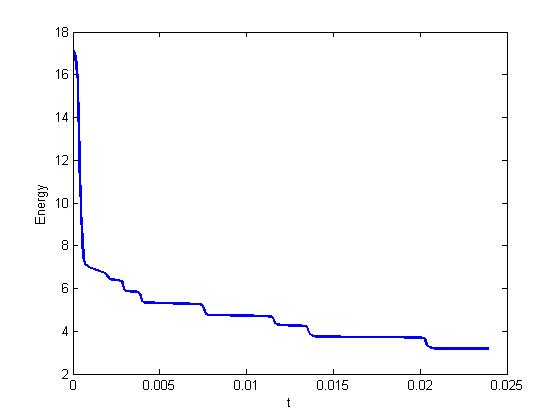}
\caption{Energy evolution of Liu-Wu model with the initial data of 0 in the bulk and 1 on the boundary (left) and the initial data of random values between -0.1 and 0.1 in the bulk and random values between 0.4 and 0.6 on the boundary (right).}
\label{lam12energy}
\end{figure}
\begin{figure}
\centering
\includegraphics[scale=0.2]{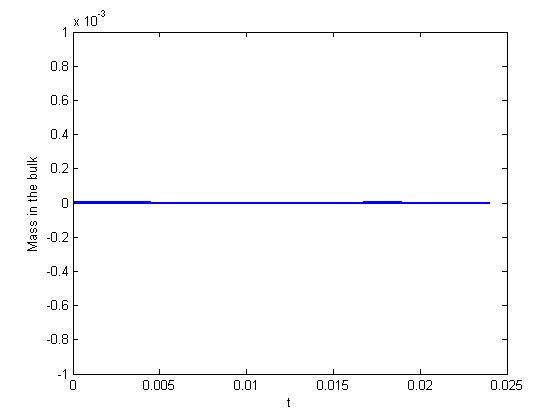}
\includegraphics[scale=0.2]{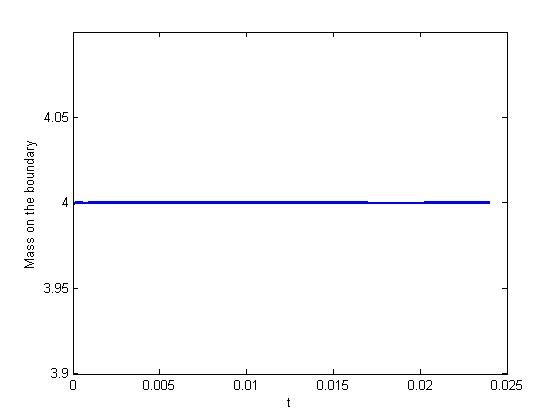}
\includegraphics[scale=0.2]{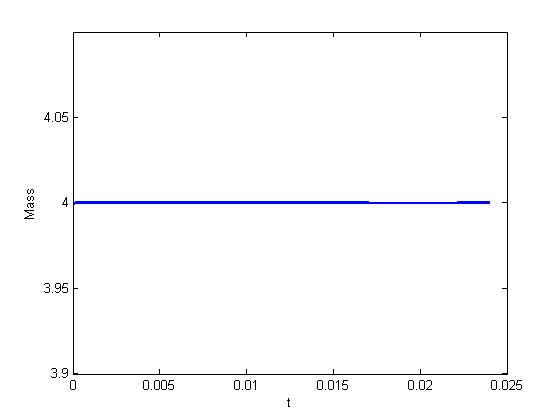}
\caption{Mass evolution of Liu-Wu model with the initial data of 0 in the bulk and 1 on the boundary: mass in the bulk(left), mass on the boundary(middle) and total mass(right).}
\label{lam1mass}
\end{figure}

\begin{figure}
\centering
\includegraphics[scale=0.2]{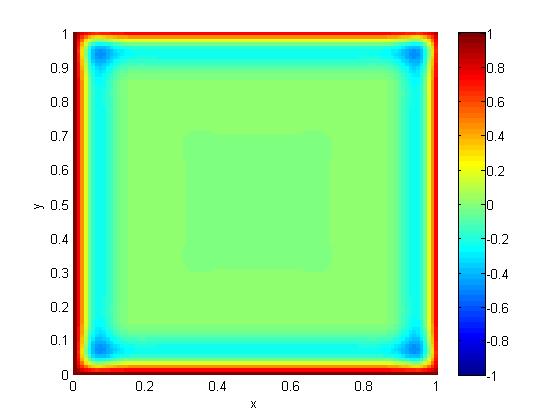}
\includegraphics[scale=0.2]{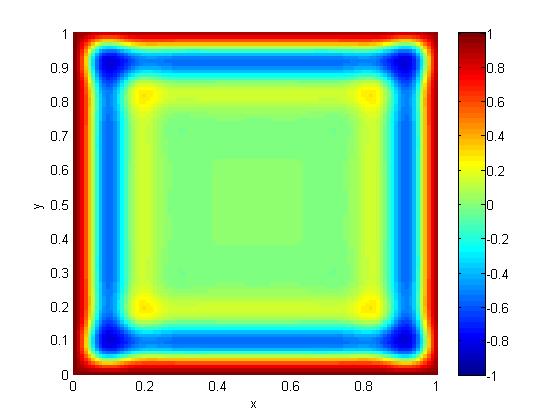}
\includegraphics[scale=0.2]{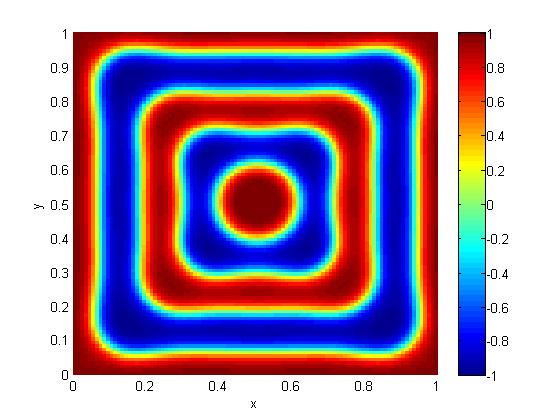}
\includegraphics[scale=0.2]{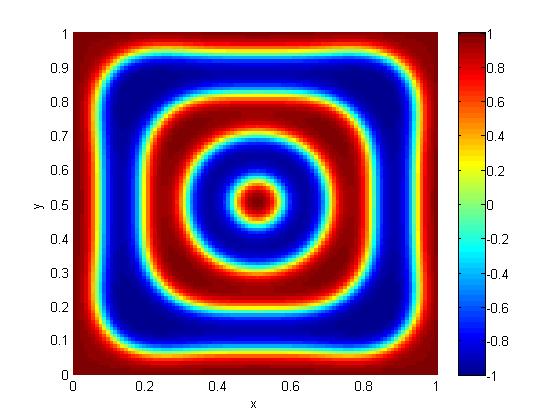}
\includegraphics[scale=0.2]{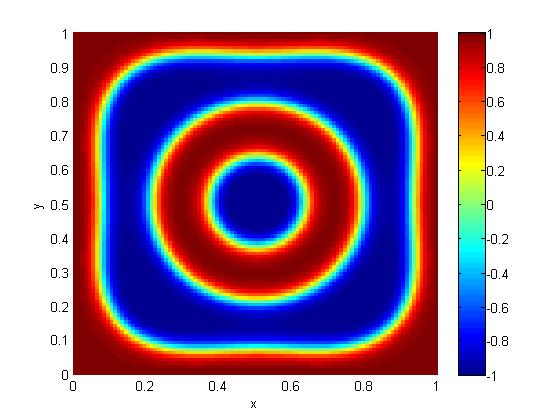}
\includegraphics[scale=0.2]{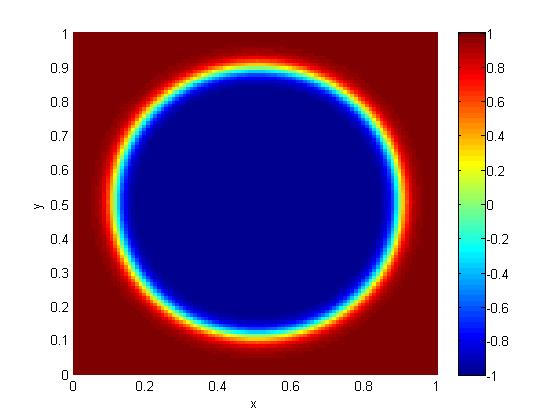}
\caption{Numerical results of Liu-Wu model with the initial data of 0 in the bulk and 1 on the boundary.}
\label{01}
\end{figure}

Secondly, the parameters are the same as those of the second numerical simulation in \cite{Garcke2020}: $\varepsilon=\delta=0.02$, $\kappa=0.075$, $\Omega=[0.5, 0.5]^2$ and the spatial step size $h=0.005$. The initial data is set as random values between -0.1 and 0.1 in the bulk and random values between 0.4 and 0.6 on the boundary. The numerical solutions after 5, 15, 50, 150, 300 and 3000 time steps are shown in Fig. \ref{lam2}. The energy evolution and the mass evolutions are plotted in Fig. \ref{lam12energy} and \ref{lam2mass}. The numerical results are consistent with \cite{Garcke2020}.

%
\begin{figure}
\centering
\includegraphics[scale=0.2]{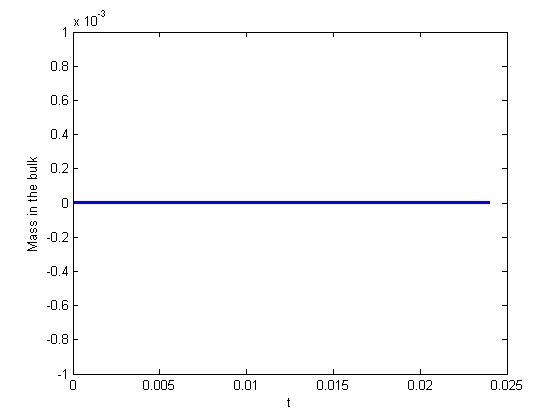}
\includegraphics[scale=0.2]{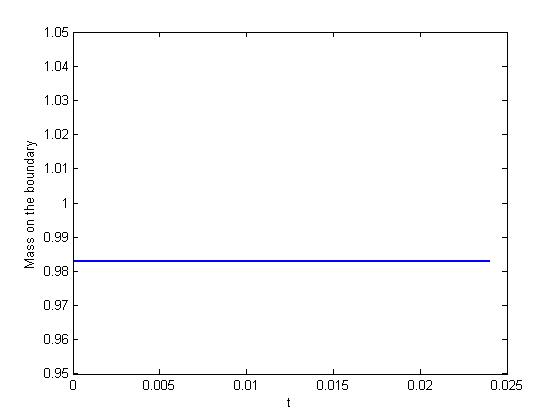}
\includegraphics[scale=0.2]{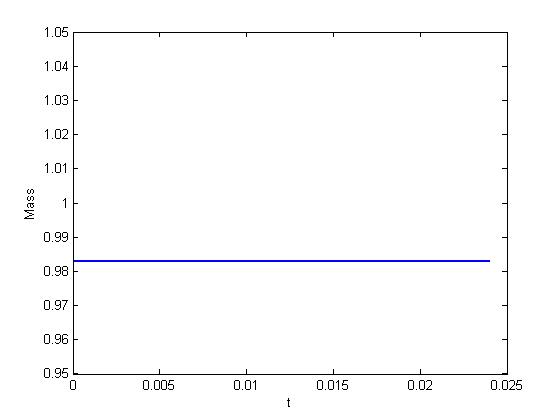}
\caption{Mass evolution of Liu-Wu model with the initial data of random values between -0.1 and 0.1 in the bulk and random values between 0.4 and 0.6 on the boundary: mass in the bulk(left), mass on the boundary(middle) and total mass(right).}
\label{lam2mass}
\end{figure}

\begin{figure}
\centering
\includegraphics[scale=0.2]{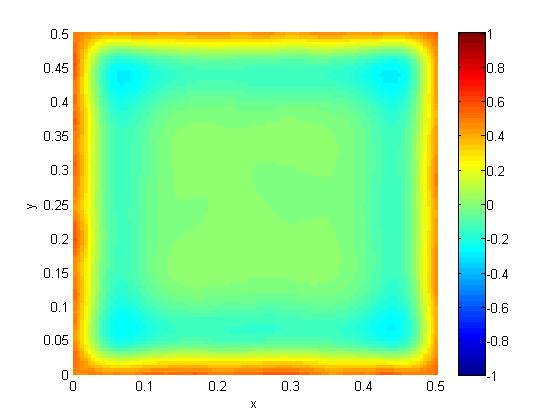}
\includegraphics[scale=0.2]{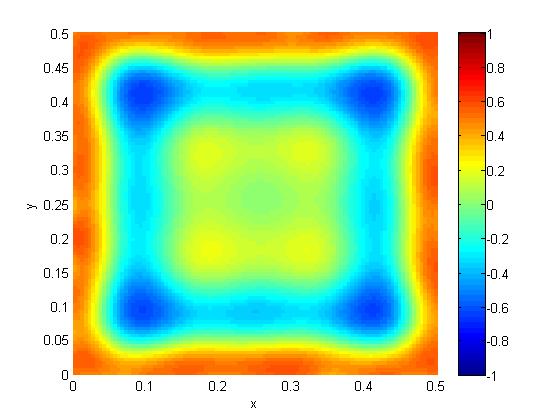}
\includegraphics[scale=0.2]{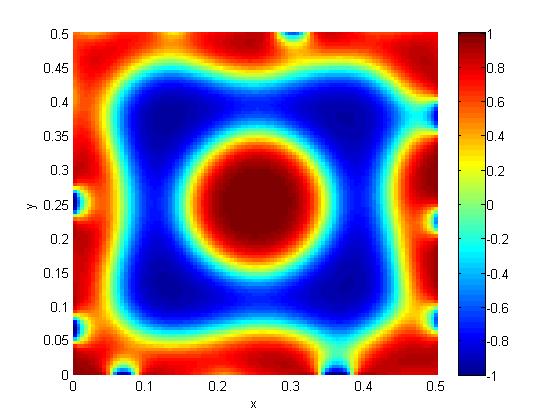}
\includegraphics[scale=0.2]{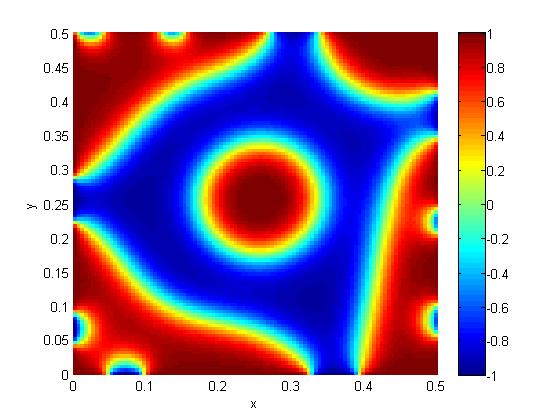}
\includegraphics[scale=0.2]{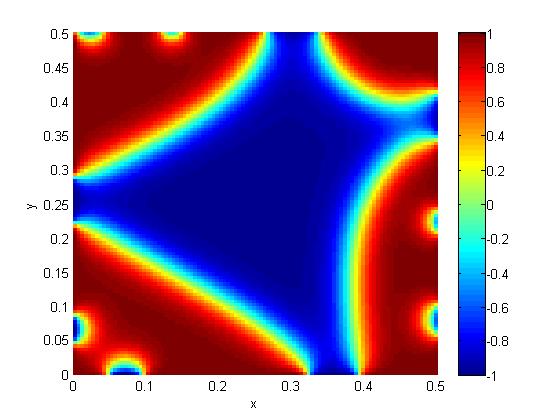}
\includegraphics[scale=0.2]{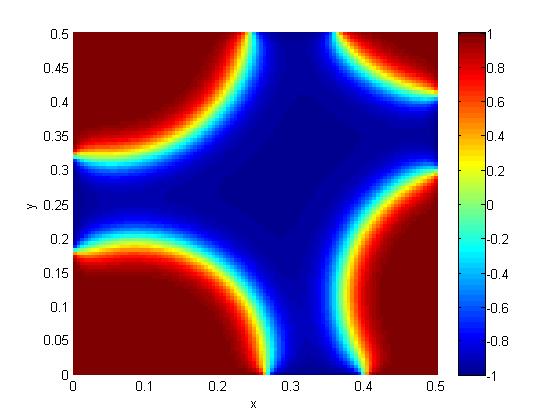}
\caption{Numerical results of Liu-Wu model with the initial data of random values between -0.1 and 0.1 in the bulk and random values between 0.4 and 0.6 on the boundary.}
\label{lam2}
\end{figure}

\subsection{Accuracy test}

We present in this section numerical accuracy test using the scheme \eqref{SIscheme1}-\eqref{SIscheme6} to support our error analysis. Let $\Omega$ to be the unit square, the spatial step size $h=0.01$ and the time steps $\tau=0.1,0.05,0.025,0.0125,6.25\times10^{-3}, 3.125\times10^{-3}$. The parameters are chosen as $\varepsilon=\delta=0.02$, $\kappa=1$ and $s_1=s_2=100$. The initial data is set to be zero in the bulk and set to be one on the boundary.
In this section, we choose $F$ and $G$ to be
\begin{equation}\label{FG}
F(\phi)=G(\phi)=\left\{\begin{aligned}
&(\phi-1)^2 \qquad \phi>1,\\
&\frac{1}4(\phi^2-1)^2 \quad -1\leq\phi<1,\\
&(\phi+1)^2 \qquad \phi\leq-1,
\end{aligned}
\right.
\end{equation}
which is modified from the classical double-well potential \eqref{classicalF}. Thus,
the Lipschitz property holds for their derivatives
\begin{equation}
\max_{\phi\in\mathbb{R}} |F''(\phi)|=\max_{\psi\in\mathbb{R}} |G''(\psi)|\leq2,
\end{equation}
which is necessary for the error estimates.

The errors are calculated as the difference between the solution of the coarse time step and that of the reference time step $\tau^*=10^{-4}$. In Fig. \ref{accuracytest} , we plot the $L^2$ errors of $\phi$ and $\psi$ between the numerical solution and the reference solution at $T = 0.5$ with different time step sizes.
The results show clearly that the convergence rate of the numerical scheme
is asymptotically first-order temporally for $\phi$ and $\psi$, which is consistent with our numerical analysis in Section \ref{s4}.

\begin{figure}
\centering
\includegraphics[scale=0.4]{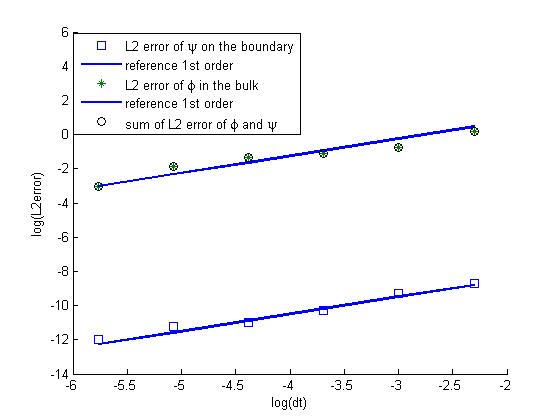}
\caption{The $L^2$ numerical errors for $\phi$ and $\psi$ at $T=0.5$.}
\label{accuracytest}
\end{figure}

\subsection{Shape deformation of a droplet}\label{droplet}

In this section, we consider the domain $\Omega=[0,1]^2$ and place a square shaped droplet with center at $(0.5, 0.25)$ and the length of each side to be 0.5 (see Fig. \ref{lam3_0}). The phase inside the droplet is set to be 1 and outside the droplet to be -1. $F$ and $G$ are chosen to be of the form \eqref{classicalF}. And the parameters are set as
$$
\varepsilon=\delta=0.02, \  \kappa=1, \  s_1=s_2=100.
$$
We simulate the behaviour of the droplet from $t = 0$ to $T =0.5$ with the time step $\tau=2\times10^{-4}$ and the spatial step size $h=0.01$.

The energy evolution and the mass evolutions are shown in Fig. \ref{lam3massandenergy}, revealing the decrease of the total energy and the conservation of mass in the bulk and on the boundary, respectively.
The time evolution of the droplet after 10, 50, 100, 500, 1000 and 2500 time steps are plotted in Fig. \ref{lam3}.
It's shown that the square shaped droplet evolves to attain the circular shape with constant mean curvature. Moreover,
the contact area of the droplet and the boundary doesn't change due to the conservation of mass, which is consistent with the previous work \cite{knopf2020}.

\begin{figure}
\centering
\includegraphics[scale=0.25]{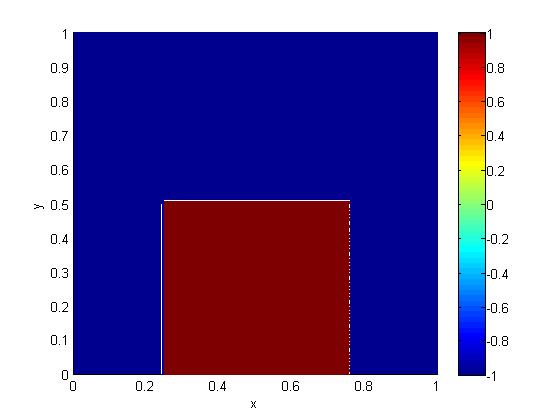}
\caption{The initial data of the square shaped droplet.}
\label{lam3_0}
\end{figure}

\begin{figure}
\centering
\includegraphics[scale=0.2]{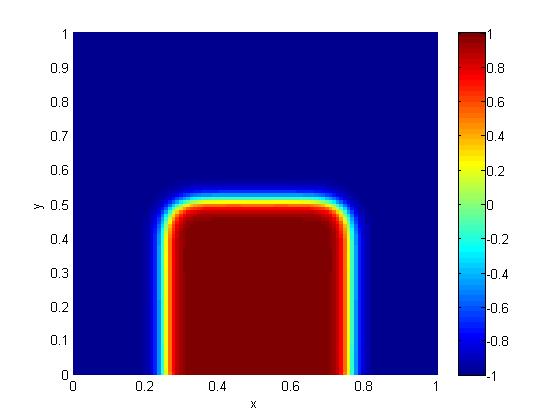}
\includegraphics[scale=0.2]{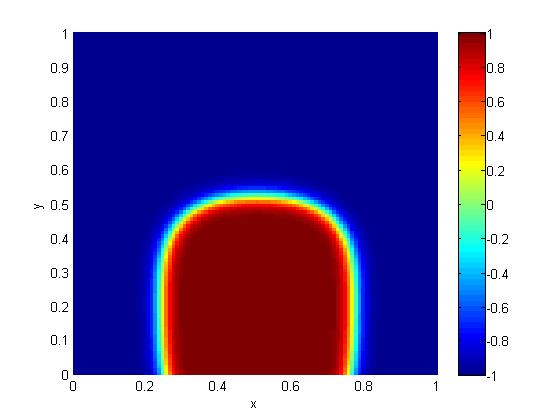}
\includegraphics[scale=0.2]{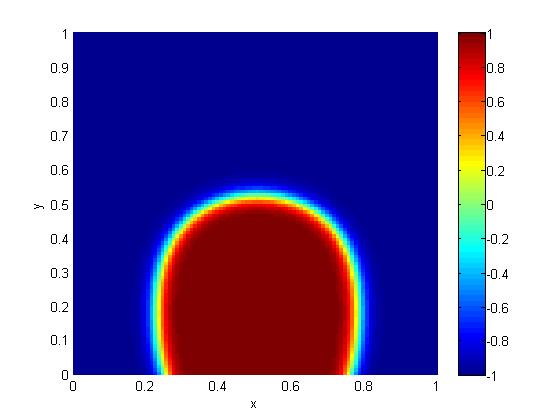}
\includegraphics[scale=0.2]{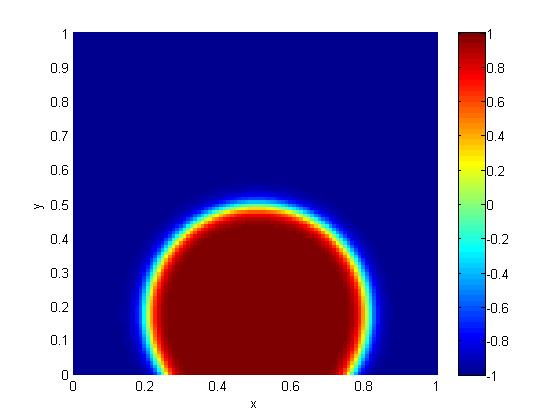}
\includegraphics[scale=0.2]{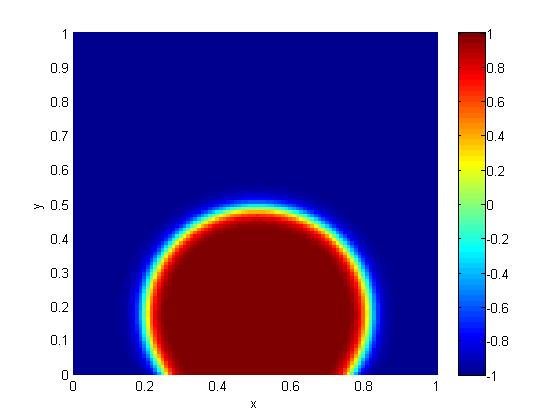}
\includegraphics[scale=0.2]{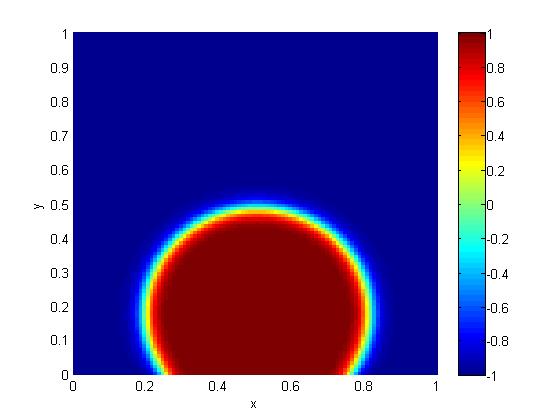}
\caption{Numerical results of Liu-Wu model with the initial data of a square shaped droplet.}
\label{lam3}
\end{figure}

\begin{figure}
\centering
\includegraphics[scale=0.33]{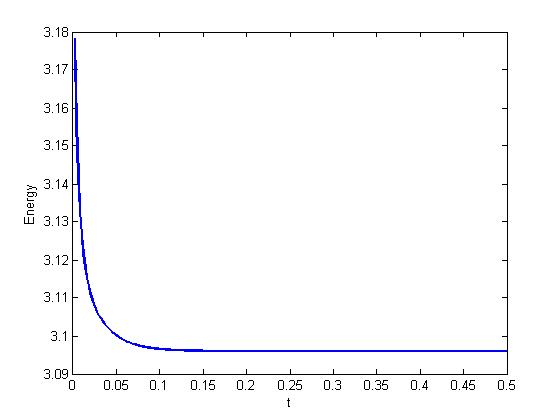}
\includegraphics[scale=0.33]{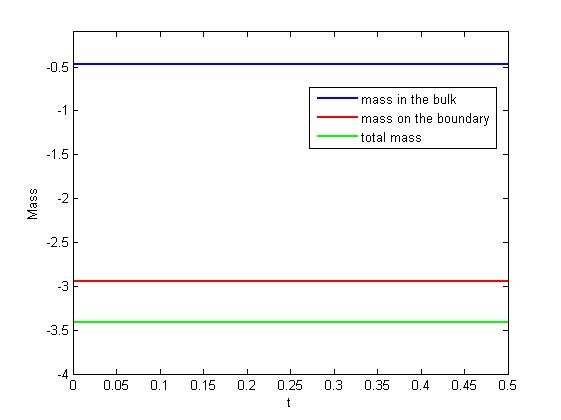}
\caption{Energy and mass evolution of Liu-Wu model with the initial data of a square shaped droplet: energy evolution(left) and mass evolution(right).}
\label{lam3massandenergy}
\end{figure}

\subsection{Comparison with different potentials}

In the numerical experiments mentioned above, the surface potential $G$ is chosen polynomial. In this section, we simulate the shape deformation of the droplet with different surface potentials. For simplicity, we denote the classical double-well potential as $G_1$, namely,
$$
G_1(\phi)=\frac{1}4(\phi^2-1)^2.
$$
And we denote the typical potential for moving contact line problems as $G_2$, namely,
$$
G_2(\phi)=\frac{\gamma}2 \cos(\theta_s)\sin(\frac{\pi}2 \phi),
$$
where $\theta_s$ stands for the static contact angle. In this section, we simulate the shape deformation of the droplet (with the initial data as in Fig. \ref{lam3_0}) with the surface potentials $G_1$ and $G_2$.

In Section \ref{droplet}, we have shown the shape deformation of the droplet, the energy and mass evolution with the surface potential $G_1$. In order to make comparisons between the two surface potentials, we choose the same parameters as those in Section \ref{droplet} with $\gamma=\frac{2\sqrt{2}}{3}$ and $\tau=1e-5$.
The energy and mass evolutions of the cases with the surface potential $G_2$ (with $\cos \theta_s=\pm \frac{1}2$) are shown in Fig. \ref{mass_60and120}, indicating
the decrease of the total energy and the conservation of mass.
The time evolution of the droplet with $G_1$ and $G_2$ after 50, 100, 200, 500, 800 and 1000 time steps are plotted in Fig. \ref{g}, \ref{60} and \ref{120}. In the case of $G_2$, the square shaped droplet also evolves to attain the circular shape, which is the same as the case that $G$ is polynomial. However, note that the contact area of the droplet and the boundary changes, which is different from the case of $G_1$. Thus, due to the conservation of mass both in the bulk and on the boundary (as shown in Fig. \ref{mass_60and120}), the values of the phase-field order parameter $\phi$ and $\psi$ are not confined in $[-1, +1]$.

\begin{figure}
\centering
\includegraphics[scale=0.2]{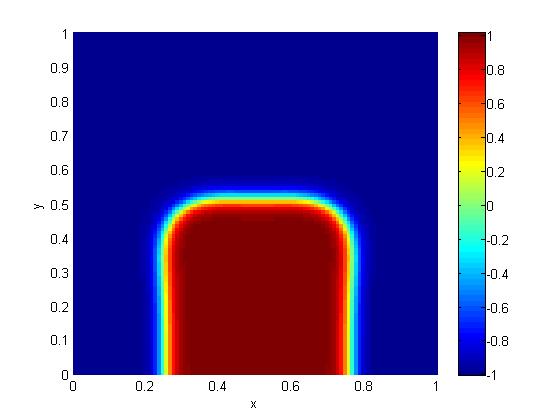}
\includegraphics[scale=0.2]{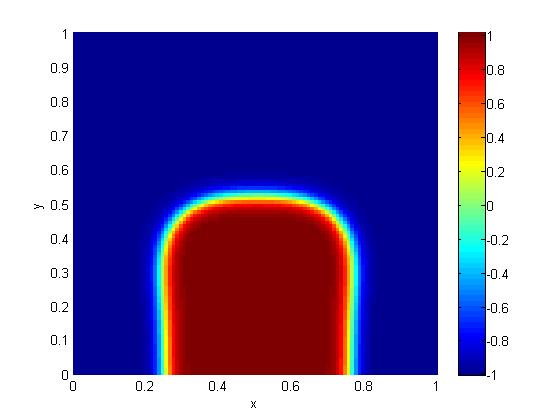}
\includegraphics[scale=0.2]{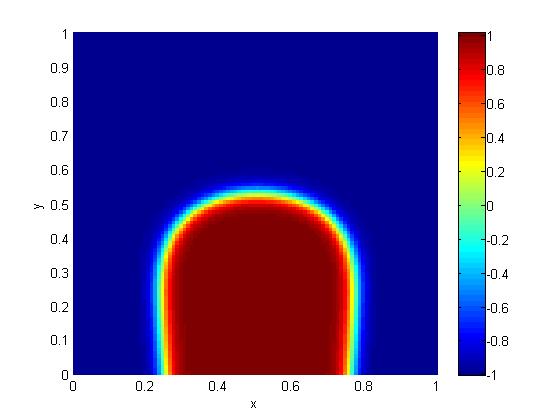}
\includegraphics[scale=0.2]{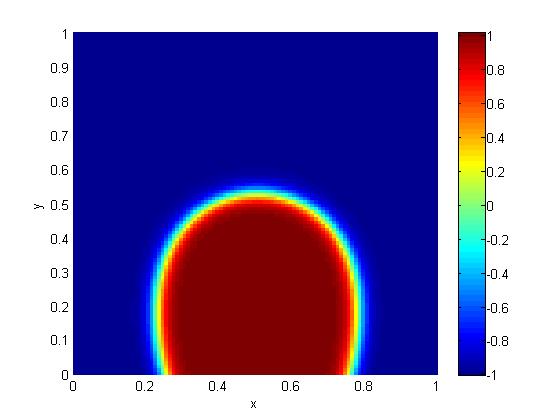}
\includegraphics[scale=0.2]{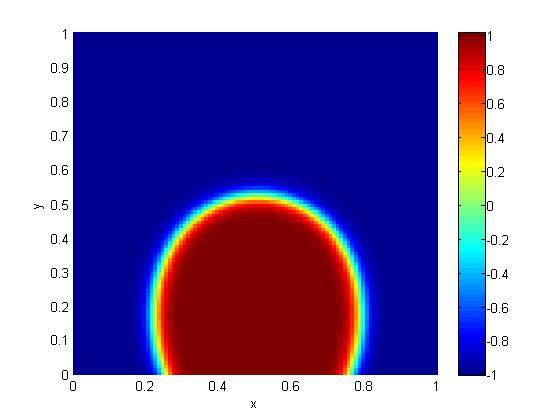}
\includegraphics[scale=0.2]{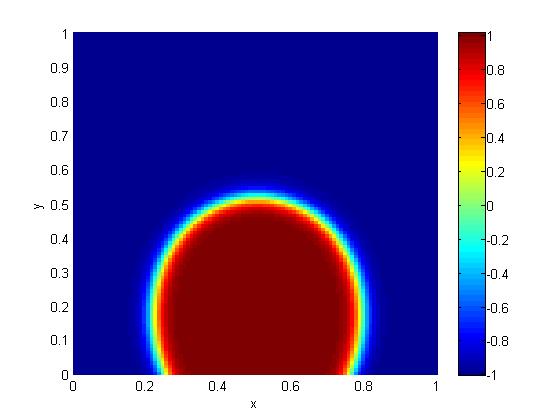}
\caption{Numerical results of Liu-Wu model with the surface potential $G_1$.}
\label{g}
\end{figure}

\begin{figure}
\centering
\includegraphics[scale=0.2]{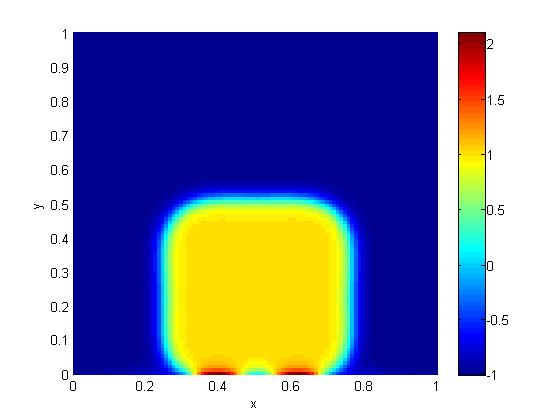}
\includegraphics[scale=0.2]{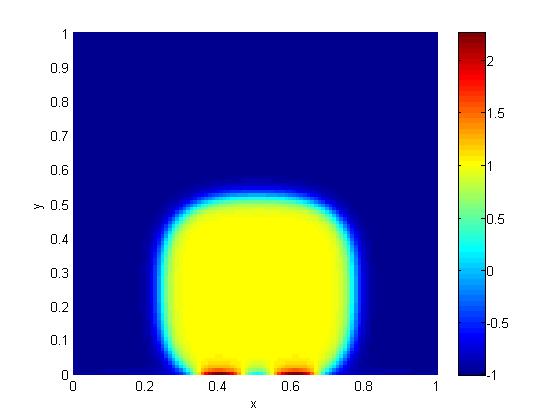}
\includegraphics[scale=0.2]{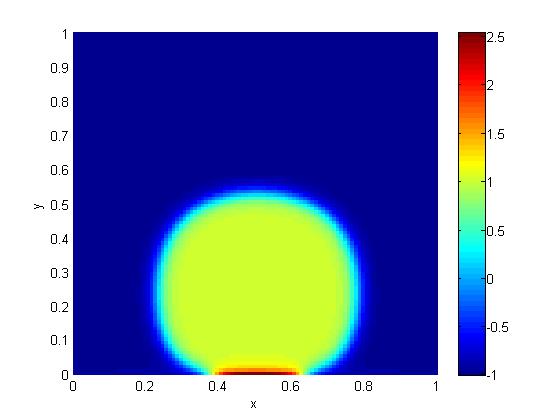}
\includegraphics[scale=0.2]{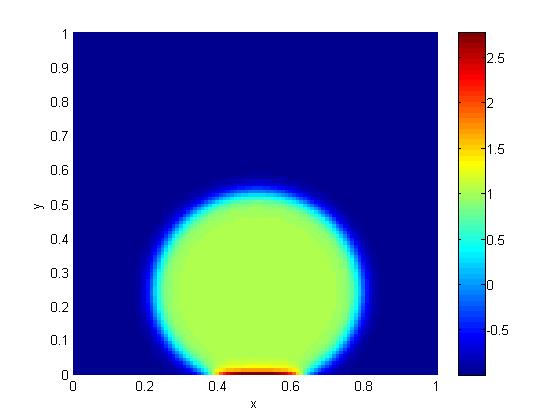}
\includegraphics[scale=0.2]{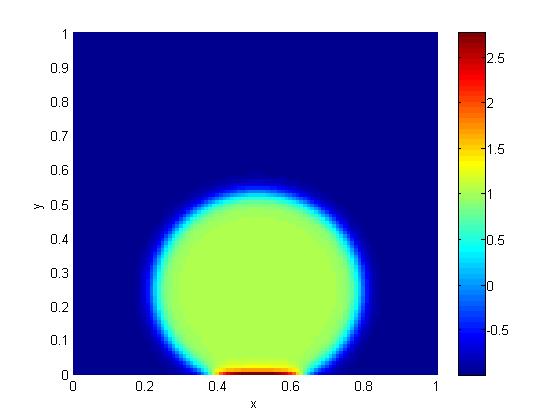}
\includegraphics[scale=0.2]{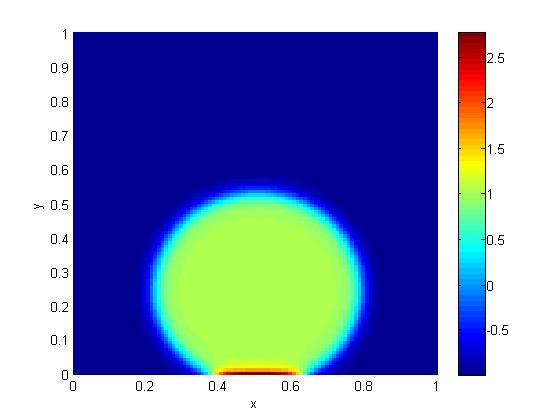}
\caption{Numerical results of Liu-Wu model with the surface potential $G_2$ ($\cos \theta_s=\frac{1}2$).}
\label{60}
\end{figure}

\begin{figure}
\centering
\includegraphics[scale=0.2]{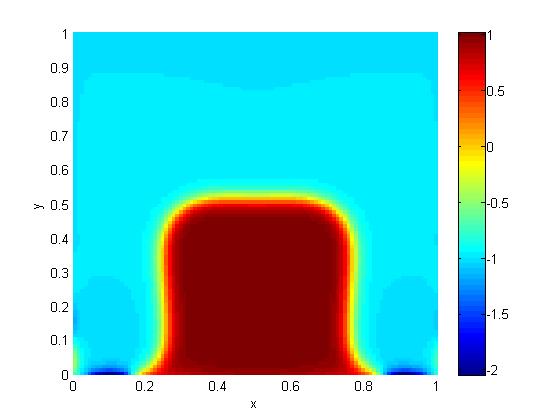}
\includegraphics[scale=0.2]{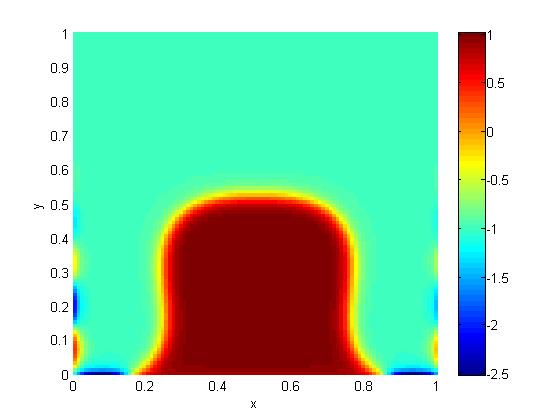}
\includegraphics[scale=0.2]{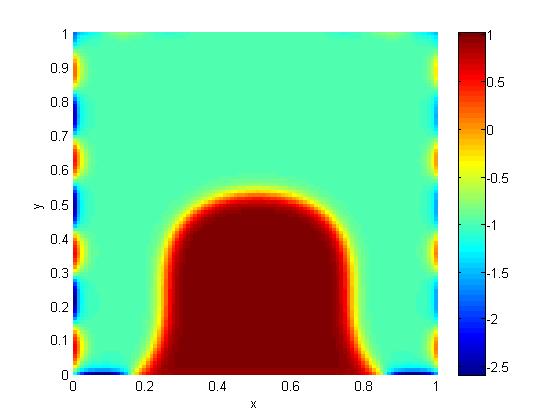}
\includegraphics[scale=0.2]{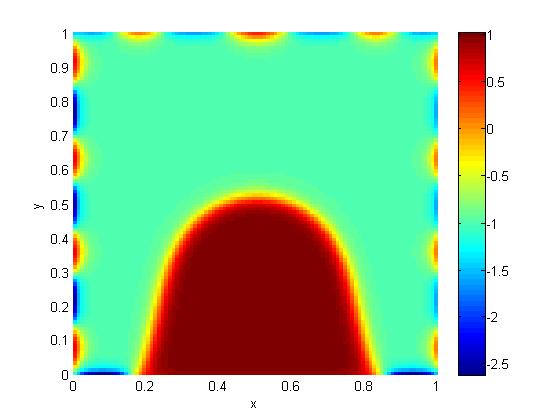}
\includegraphics[scale=0.2]{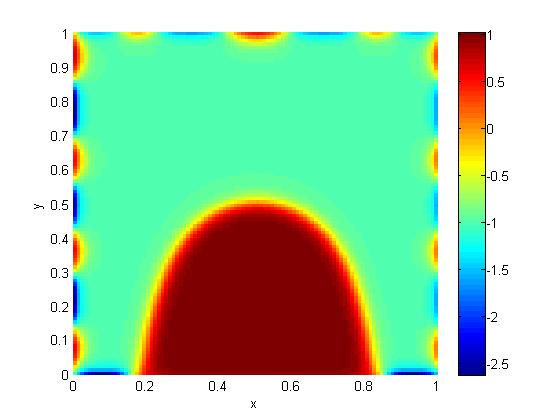}
\includegraphics[scale=0.2]{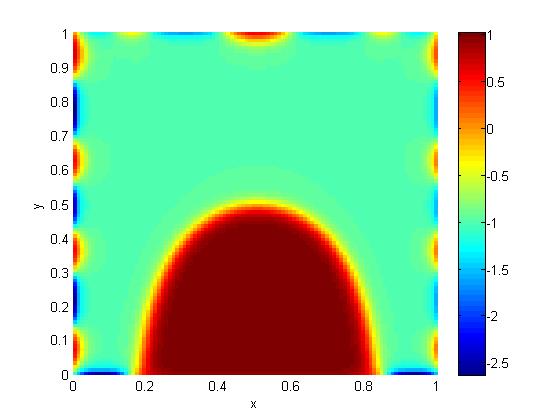}
\caption{Numerical results of Liu-Wu model with the surface potential $G_2$ ($\cos \theta_s=-\frac{1}2$).}
\label{120}
\end{figure}

\begin{figure}
\centering
\includegraphics[scale=0.2]{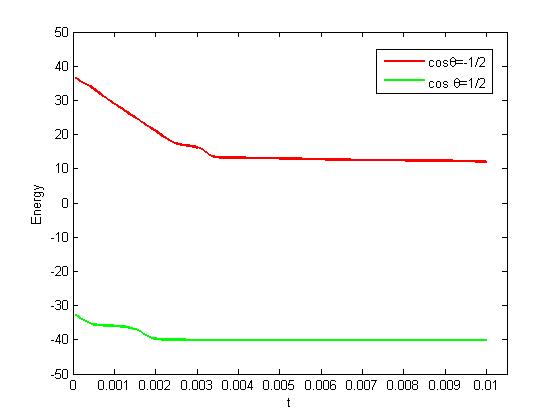}
\includegraphics[scale=0.2]{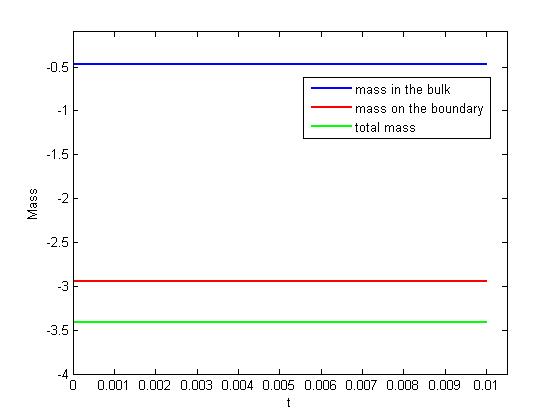}
\includegraphics[scale=0.2]{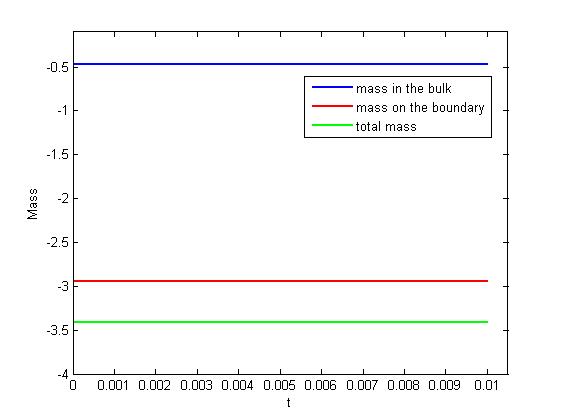}
\caption{Energy evolution of Liu-Wu model with the surface potential $G_2$ (left). Mass evolution of Liu-Wu model with the surface potential $G_2$: $\cos \theta_s=\frac{1}2$ (middle) and $\cos \theta_s=-\frac{1}2$ (right).}
\label{mass_60and120}
\end{figure}

\section{Conclusions}

In the present work, we consider numerical approximations for the Cahn-Hilliard equation with dynamic boundary conditions (C. Liu et. al., Arch. Rational Mech. Anal., 2019).
To solve the model, we develop an efficient scheme based on the stabilized linearly implicit approach, which is  first-order in time, linear and energy stable. The stabilization terms are used to enhance the stability of the scheme.
To the best of our knowledge, this is the first linear and energy stable scheme for solving the Liu-Wu model. The semi-discretized-in-time error estimates for the scheme are also derived.
The energy stability and the accuracy of the developed scheme are demonstrated numerically
by constructing numerical experiments, including the comparison with the former work, accuracy tests with respect to the time step size and the shape deformation of a droplet.

\vspace{1cm}

\begin{center}
Acknowledgment
\end{center}
The authors would like to thank Prof. Chun Liu for some useful discussions on the subject of this article. X. Bao is thankful to Prof. Chun Liu, Prof. Yiwei Wang and Prof. Qing Cheng for some stimulating discussions during the visit of Illinois Institute of Technology.
X. Bao is partially supported by China Scholarship Council (No. 201906040019). H. Zhang was partially supported by the National Natural Science Foundation of China (Nos. 11971002 and 11471046).
%

\end{document}